\documentclass[reqno, 11pt]{amsart}
\usepackage[utf8]{inputenc}
\usepackage[dvips]{graphicx,color}
\usepackage[all]{xy}
\usepackage{enumerate}
\usepackage[english]{babel}
\usepackage{graphicx,pstricks,pst-plot}
\usepackage{xcolor} 
\usepackage[colorlinks=true, linkcolor=blue, citecolor=green, urlcolor=purple]{hyperref}
\usepackage{setspace}
\usepackage{cancel}
\usepackage[left=2.5cm, right=2.5cm, bottom=2cm]{geometry}
\usepackage{amsmath, amssymb, amsthm, epsfig, mathtools}
\usepackage{latexsym}
\usepackage[nameinlink,capitalise]{cleveref}
\usepackage{url} 
\usepackage{bbm}

\DeclarePairedDelimiter\floor{\lfloor}{\rfloor}

\def\today{\ifcase\month\or
  January\or February\or March\or April\or May\or June\or
  July\or August\or September\or October\or November\or December\fi
  \space\number\day, \number\year}

\usepackage{stmaryrd}
\providecommand{\ran}[2]{\llbracket {#1}, {#2} \rrbracket}

\newcommand{\w}[1]{\widehat{#1}} 
\newcommand{\bd}[1]{\boldsymbol{#1}}
\newcommand{\obs}{\mathrm{Obs}(w)_{s, \ld, \xi}}
\DeclareMathOperator{\tc}{tc}
\DeclareMathOperator{\pw}{pw}

\newcommand{\by}{\overline{y}}
\newcommand{\dom}{\text{dom}}
\newcommand{\weta}{\widehat{\eta}}
\newcommand{\ueta}{\breve{\eta}}

\newcommand{\wal}{\widehat{\alpha}}
\newcommand{\ual}{\breve{\alpha}}
\newcommand{\mc}{\mathcal}
\newcommand{\A}{\mc{A}}
\newcommand{\B}{\mc{B}}
\newcommand{\D}{\mc{D}}
\newcommand{\F}{\mc{F}}
\newcommand{\G}{\mc{G}}
\def\H{\mathcal H}
\newcommand{\J}{\mc{J}}
\newcommand{\M}{\mc{M}}
\newcommand{\X}{\mc{X}}
\newcommand{\R}{\mathbb{R}}
\newcommand{\N}{\mathbb{N}}
\newcommand{\bw}{\boldsymbol{w}}
\newcommand{\bA}{\boldsymbol{A}}
\newcommand{\bX}{\boldsymbol{X}}
\newcommand{\kB}{\mathfrak{B}}
\newcommand{\kD}{\mathfrak{D}}
\newcommand{\tr}{\text{tr}}
\newcommand{\ld}{\lambda}
\newcommand{\al}{\alpha}
\newcommand{\veps}{\varepsilon}
\newcommand{\vp}{\varphi}
\newcommand{\wxi}{\tilde{\xi}_{s, \ld}}

\providecommand{\norm}[1]{\lVert#1\rVert}

\providecommand{\inn}[1]{\langle#1\rangle}

\newtheorem{theorem}{Theorem}[section]
\newtheorem{lemma}[theorem]{Lemma}
\newtheorem{proposition}[theorem]{Proposition}

\theoremstyle{definition}
\newtheorem{definition}[theorem]{Definition}

\theoremstyle{remark}
\newtheorem{remark}[theorem]{Remark}
\newtheorem{assum}{Hypothesis} 
\newenvironment{assump}[2][]
  {\renewcommand\theassum{#2}\begin{assum}[#1]}
  {\end{assum}}

\newcommand{\intav}[1]{\mathchoice {\mathop{\vrule width 6pt height 3 pt depth  -2.5pt
\kern -8pt \intop}\nolimits_{\kern -6pt#1}} {\mathop{\vrule width
5pt height 3  pt depth -2.6pt \kern -6pt \intop}\nolimits_{#1}}
{\mathop{\vrule width 5pt height 3 pt depth -2.6pt \kern -6pt
\intop}\nolimits_{#1}} {\mathop{\vrule width 5pt height 3 pt depth
-2.6pt \kern -6pt \intop}\nolimits_{#1}}}

\begin{document}

\title[Carleman estimates for the KdV with piecewise constant main coefficient]{Carleman estimates for the Korteweg-de Vries equation with piecewise constant main coefficient}
\author[]{Cristóbal Loyola}
\date{\today}
\address{Sorbonne Université, Université Paris Cité, CNRS, Laboratoire Jacques-Louis Lions, LJLL, F-75005 Paris, France}
\email{cristobal.loyola@sorbonne-universite.fr}

\allowdisplaybreaks
\numberwithin{equation}{section}

\begin{abstract}
    In this article, we investigate observability-related properties of the Korteweg-de Vries equation with a discontinuous main coefficient, coupled by suitable interface conditions. The main result is a novel two-parameter Carleman estimate for the linear equation with internal observation, assuming a monotonicity condition on the main coefficient. As a primary application, we establish the local exact controllability to the trajectories by employing a duality argument for the linear case and a local inversion theorem for the nonlinear equation. Secondly, we establish the Lipschitz-stability of the inverse problem of retrieving an unknown potential using the Bukhge{\u\i}m-Klibanov method, when some further assumptions on the interface are made. We conclude with some remarks on the boundary observability.
\end{abstract}

\maketitle
\setcounter{tocdepth}{1}
\tableofcontents

\section{Introduction}

Let $T>0$, $L>0$ and $p: [0, L]\to\R^+$ be a piecewise constant function where $p(x)=p_{k}>0$ on each $[a_{k}, a_{k+1})$ with $\Gamma=\{a_1<a_2<\cdots<a_{N-1}\}$ and $a_0=0$, $a_N=L$. Given some initial data $y_0$, we study the following Korteweg-de Vries equation with piecewise constant main coefficient
\begin{align}\label{eq:NLKdV}
\left\{
\begin{array}{rll}
y_t+p(x)y_{xxx}+y_x+yy_x=0, &\ (t, x)\in (0, T)\times (0, L),\\
y(t, 0)=y(t, L)=y_x(t, L)=0, &\ t\in (0, T),\\
y(0, x)=y_0(x), &\ x\in (0, L),
\end{array}
\right.
\end{align}
coupled by the \emph{transmission conditions}
\begin{align}\label{eq:tc}\tag{TC}
    \left\{\begin{array}{rrlll}
    y(t, a_k^-)&=&y(t, a_k^+), & ~t\in (0, T),~k\in\ran{1}{N-1},\\
    \sqrt{p_{k-1}}y_x(t, a_k^-)&=&\sqrt{p_k}y_x(t, a_k^+), & ~t\in (0, T),~k\in\ran{1}{N-1},\\
    p_{k-1}y_{xx}(t, a_k^-)&=&p_ky_{xx}(t, a_k^+), & ~t\in (0, T),~k\in\ran{1}{N-1},
    \end{array}\right.
\end{align}
where $\ran{0}{N-1}:=\{0, \ldots, N-1\}$. Despite the discontinuity of the coefficient $p$, the transmission conditions allow us to consider this model as a whole in $(0, T)\times (0, L)$, since they act as boundary conditions on each $(a_k, a_{k+1})$, $k\in\ran{0}{N-1}$.

The Korteweg-de Vries equation (KdV) is a well-known dispersive equation, introduced by Korteweg-de Vries \cite{KdV1895} to model the propagation of water waves in a shallow channel, namely, water waves with small amplitude and large wavelength compared to the undisturbed depth profile. In this context, this kind of water waves in a channel with a sudden jump in the depth profile have been modeled by a KdV equation with a discontinuous main coefficient \cite{PV92}. Thus, in this model the function $p=p(x)$ can be interpreted as the undisturbed depth profile with jumps of the channel, whereas $p(x)+y(t, x)$ represents the wave surface at time $t$ at position $x$. On the mathematical side, to solve the KdV equation on the half-line, Deconinck, Sheils and Smith \cite{DSS16} has proposed (several) interface conditions involving the first three derivatives (in space) of the solution at the interface. The specific interface conditions of system \eqref{eq:NLKdV}-\eqref{eq:tc} were proposed by Crépeau \cite{Crepeau16}, where the exact boundary controllability of such a model is studied. In this work, we continue the study of this equation by obtaining a new Carleman estimate under a monotonicity hypothesis on the main coefficient, allowing us to deduce some results on the controllability and the recovery of some parameter for this equation. 

\subsection{Main results} Let us define $I_{k}:=(a_k, a_{k+1})$, $k\in\ran{0}{N-1}$. In what follows $\omega$ will be always a non-empty open subset of $(0, L)$. Due to the discontinuity of the main coefficient, it is natural to impose some restriction on where the observation zone should be located. In this direction we will introduce the following assumption on the pair $(\omega, p)$:

\begin{assump}{$\mathfrak{M}$}\label{assumM} there exists $j\in \ran{0}{N-1}$ such that the observation zone $\omega$ is such that $\overline{\omega}\subset I_j$, henceforth denoted as $\omega\Subset I_j$, and the following monotonicity property holds depending on the value of $j$:
    \begin{enumerate}
        \item if $j\not\in \{0, N-1\}$ then
        \begin{align*}
        \begin{cases}
            p(a_k^-)>p(a_k^+),\ k\in\ran{1}{j},\\
            p(a_k^-)<p(a_k^+),\ k\in\ran{j+1}{N-1};
        \end{cases}
        \end{align*}
        \item if $j=0$ then
        \begin{align*}
            p(a_k^-)<p(a_k^+),\ k\in\ran{1}{N-1};
        \end{align*}
        \item if $j=N-1$ then
        \begin{align*}
            p(a_k^-)>p(a_k^+),\ k\in\ran{1}{N-1}.
        \end{align*}
    \end{enumerate}
\end{assump}
Let $s\geq 0$ and let us introduce the space
\begin{align}\label{S1:def:H3gamma}
    H_{\pw}^s(0, L):=\{u\in L^2(0, L)\ |\ u_{|_{I_{k}}}\in H^s(I_k),\ k\in\ran{0}{N-1}\}.
\end{align}
Via isomorphism, $H_{\pw}^s(0, L)$ can be seen as the direct sum of the Sobolev spaces $H^s(I_k)$ for $k\in\ran{0}{N-1}$. Thus, it has a Hilbert space structure equipped the inner product
\begin{align*}
    \inn{u, v}_{H_{\pw}^s(0, L)}=\sum_{k=0}^{N-1} \inn{u_{|_{I_{k}}}, v_{|_{I_{k}}}}_{H^s(I_k)}.
\end{align*}

\subsubsection{Carleman estimate}  Let $Q:=(0, T)\times (0, L)$. For some $\rho_0$, $\rho_1>0$ given, let us assume that
\begin{align*}
    \rho_0\leq \min_{x\in [0, L]}p(x)\ \text{ and }\ \max_{x\in [0, L]}p(x)\leq \rho_1.
\end{align*}
Let $b$, $d\in L^\infty(Q)$ and let $\mc{L}: \mc{V}\to L^2(Q)$ be the differential operator given by
\begin{gather}\label{S1:def:opL}
\mc{L}=\partial_t+p(x)\partial_{x}^3+b(t, x)\partial_x+d(t, x),
\end{gather}
where $\mc{V}$ is the space of functions $u\in L^2(0, T; H_{\pw}^3(0, L)\cap H_0^1(0, L))$ such that $\mc{L}u\in L^2(Q)$ and $u$ satisfies the transmission conditions
\begin{align}\label{S1:eq:tc-carleman}
    \left\{\begin{array}{rrlll}
    u(t, a_k^-)&=&u(t, a_k^+), & ~t\in (0, T),~k\in\ran{1}{N-1},\\
    \sqrt{p_{k-1}}u_x(t, a_k^-)&=&\sqrt{p_k}u_x(t, a_k^+), & ~t\in (0, T),~k\in\ran{1}{N-1},\\
    p_{k-1}u_{xx}(t, a_k^-)&=&p_ku_{xx}(t, a_k^+), & ~t\in (0, T),~k\in\ran{1}{N-1}.
    \end{array}\right.
\end{align}

Consider $\omega_0$ a non-empty open subset of $\omega$ such that $\omega_0\Subset \omega$. Fix $\kappa\in (1, 2)$ and let $\beta$ be the weight function constructed by \cref{S3:lem:weightObs} below. For $\ld>0$, we introduce the Carleman weights
\begin{align}\label{S1:intro:weights}
    0<\eta(t, x)=\dfrac{e^{\kappa\ld\norm{\beta}_{\infty}}-e^{\ld\beta(x)}}{t(T-t)}\ \text{   and   }\   \xi(t, x)=\dfrac{e^{\ld\beta(x)}}{t(T-t)},
\end{align}
for $(t, x)\in Q$, The main result is the following two-parameter Carleman estimate.

\begin{theorem}\label{thm:carleman_omega}
Let $(\omega, p)$ satisfy Hypothesis \ref{assumM} and let $\omega_0\Subset \omega$ be non-empty and open. There exist $s_0>0$, $\ld_0>0$ and a constant $C>0$ depending on $\omega$, $\Gamma$, $L$, $T$, $p$, $\norm{\beta}_{C^3([0, L]\setminus\Gamma)}$, $s_0$ and $\ld_0$ such that for all $u\in \mc{V}$ we have
\begin{multline}\label{thm:ineq:carleman_omega}
C\iint_Q e^{-2s\eta}\big(s^5\ld^6\xi^5|u|^2+s^3\ld^4\xi^3|u_x|^2+s\ld^2\xi|u_{xx}|^2\big)dxdt\\ \leq \norm{e^{-s\eta}\mc{L}u}_{L^2(Q)}^2+\iint_{(0, T)\times\omega}e^{-2s\eta} \big(s^5\ld^6 \xi^5|u|^2+s^3\ld^4\xi^3|u_x|^2+s\ld^2 \xi|u_{xx}|^2\big)dxdt
\end{multline}
for any $s\geq s_0$ and $\ld\geq \ld_0.$
\end{theorem}

The above estimate is derived by following Fursikov and Imanuvilov \cite{FI96}. We emphasize that a Carleman estimate is of interest in itself due to the many applications they have found. We now give two applications and in \cref{S6} we briefly discuss the case of boundary observability.

\subsubsection{Local controllability to the trajectories} Let us consider the controlled KdV equation
\begin{align}\label{eq:NLKdV-control}
\left\{
\begin{array}{rll}
y_t+p(x)y_{xxx}+y_x+yy_x=\mathbbm{1}_\omega v, &\ (t, x)\in (0, T)\times (0, L),\\
y(t, 0)=y(t, L)=y_x(t, L)=0, &\ t\in (0, T),\\
y(0, x)=y_0(x), &\ x\in (0, L),
\end{array}
\right.
\end{align}
coupled by the transmission conditions \eqref{eq:tc}, where $y_0$ is the initial condition, and $v$ is a control localized in some non-empty open set $\omega\subset (0, L)$. We are interested in the \emph{exact controllability to the trajectories} for the KdV equation \eqref{eq:NLKdV-control}. More precisely, we wonder if, given $T>0$ and a solution $\by$ of the uncontrolled KdV equation
\begin{align}\label{eq:NLKdV-traj}
\left\{
\begin{array}{rll}
\by_t+p(x)\by_{xxx}+\by_x+\by\by_x = 0, & ~(t, x)\in (0, T)\times (0, L),\\
\by(t, 0)=\by(t, L)=\by_{x}(t, L)=0, & ~t\in (0, T),\\
\by(0, x) =  \by_0(x), & ~x\in (0, L),
\end{array}
\right.
\end{align}
coupled by the corresponding transmission conditions \eqref{eq:tc}, there exists a control $v=v(t, x)$ such that the corresponding solution $y=y(t, x)$ satisfies $y(T, \cdot)=\by(T, \cdot)$ on $(0, L)$. In \cref{S2} we discuss the well-posedness of systems \eqref{eq:NLKdV-control} and \eqref{eq:NLKdV-traj}, both coupled by \eqref{eq:tc}. The controllability result is the following.

\begin{theorem}\label{thm:NLKdV-control}
    Let $T>0$ and let $(\omega, p)$ satisfy Hypothesis \ref{assumM}. If $\by\in C([0, T], L^2(0, L))\cap L^2(0, T; H^1(0, L))$ is the solution of \eqref{eq:NLKdV-traj}, then there exists $\delta>0$ such that for any $y_0\in L^2(0, L)$ satisfying $\norm{y_0-\by_0}_{L^2(0, L)}\leq \delta$, we can find a control $v\in L^2((0, T)\times \omega)$ such that the corresponding solution $y$ to \eqref{eq:NLKdV-control} satisfies
    \begin{align*}
        y(T, \cdot)=\by(T, \cdot)\ \text{ in }\ (0, L).
    \end{align*}
\end{theorem}

The strategy consists in first consider the system satisfied by $z=y-\by$, which is given by
\begin{align}\label{eq:NLKdV-control-dif}
\left\{
\begin{array}{rll}
z_t+p(x)z_{xxx}+z_x+(\by z)_x+zz_x = \mathbbm{1}_\omega v, &\ (t, x)\in (0, T)\times (0, L),\\
z(t, 0)=z(t)=z_x(t, L)=0, &\ t\in (0, T),\\
z(0, x) =  z_0(x), &\ x\in (0, L),
\end{array}
\right.
\end{align}
coupled by the corresponding transmission conditions \eqref{eq:tc}. Then, establishing the exact control to the trajectories reduces to establish the null controllability of the system \eqref{eq:NLKdV-control-dif}. This is done by studying the null controllability of the linearization of \eqref{eq:NLKdV-control-dif} and then employing a fixed point argument to treat the nonlinear system. The null controllability of the linearized system follows by a duality argument and a suitable observability estimate for the adjoint system. This observability estimate is derived from the Carleman estimate given in \cref{thm:carleman_omega}.

\subsubsection{Retrieving a potential term} Consider the nonlinear KdV equation with potential $\mu=\mu(x)$,
\begin{align}\label{S5:eq:KdVpotential}
\left\{
\begin{array}{rll}
    y_t+p(x)y_{xxx}+\mu(x) y_x+yy_x=0, & ~(t, x)\in (0, T)\times (0, L),\\
    y(t, 0)=h_1(t),\ y(t, L)=h_2(t),\ y_{x}(t, L)=h_3(t), & ~t\in (0, T),\\
    y(0, x)=y_0(x), & ~x\in (0, L),
\end{array}
\right.
\end{align}
coupled by the transmission conditions \eqref{eq:tc}, initial condition $y_0$ and boundary data $\vec{h}=(h_1, h_2, h_3)$. We denote its solution by $y=y[\mu]$. We make the following assumptions on the interface $\Gamma$.

\begin{assump}{$\mathfrak{I}$}\label{assumG}
    The interface $\Gamma$ and the coefficient $p$ satisfy:
    \begin{itemize}
        \item The middle point of the domain is not an interface point: $a_k\neq L/2$ for each $k\in\ran{0}{N-1}$.
        \item The interface is symmetric, in the sense that $a_k+a_{N-k}=L$ for all $k\in\ran{0}{N}$.
    \end{itemize}
\end{assump}

For $m>0$ given, let us introduce the set of admissible potentials
\begin{align*}
    \mathfrak{P}_{\leq m}^{adm}=\left\{\mu\in H_{\pw}^3\cap H_0^1(0, L)\ \Big|\ \begin{array}{c} \mu'(0)=0,\ \mu(a_k)=0 \text{ and }\ [\mu']_{a_k}=[\sqrt{p}\mu'']_{a_k}=0,\ k\in\ran{1}{N-1} \\ \mu(x)=\mu(L-x),\forall x\in [0, L] \text{ and } \norm{\mu}_{H_{\pw}^3(0, L)}\leq m\end{array}\right\}.
\end{align*}
Following the Bukhge{\u\i}m-Klibanov method, we can employ a slight variant of the Carleman estimate given in \cref{thm:carleman_omega} to establish the Lipschitz continuity of the inverse problem consisting on retrieving the potential term on the equation \eqref{S5:eq:KdVpotential}. The result is the following.

\begin{theorem}\label{S1:thm:IPpotential}
Let $\omega\subset (0, L)$ be a nonempty open set containing $L/2$. Assume that $\Gamma$ satisfy Hypothesis \ref{assumG}. Let $p$ be symmetric with respect to $L/2$, that is, $p(x)=p(L-x)$ for all $x\in [0, L]$ and
\begin{align*}
        p(a_k^-)>p(a_k^+),\ k\in\ran{1}{\floor{N/2}},\ \text{ and }\
        p(a_k^-)<p(a_k^+),\ k\in\ran{\floor{N/2}}{N}.
\end{align*}
Let $m$, $r_0$ and $K$ be some given positive constants. For any $\mu\in \mathfrak{P}_{\leq m}^{adm}$, let $(y_0, \vec{h})\in \mc{Z}_{6, T}$ be compatible data (in the sense of \cref{def:compatible} below) with respect to $\mu$, with $y_0'(x)=y_0'(L-x)$ for $x\in [0, L]$. If $\min_{x\in [0, L]}|y_0'(x)|\geq r_0>0$, there exists a positive constant $C$ depending on $L$, $T$, $\Gamma$, $\omega$, $m$, $r_0$ and $K$ such that for any $\nu\in \mathfrak{P}_{\leq m}^{adm}$,
\begin{align*}
    \norm{\mu-\nu}_{L^2(0, L)}\leq C\norm{y-z}_{H^1(0, T; H_{\pw}^2(\omega))},
\end{align*}
where the solutions $y=y[\mu]$ and $z=z[\nu]$ of \eqref{S5:eq:KdVpotential}-\eqref{eq:tc} issued from $(y_0, \Vec{h})$ satisfy
\begin{align*}
    \max\{\norm{y}_{W^{1,\infty}(0, T;W^{1,\infty}(0, L))}, \norm{z}_{W^{1,\infty}(0, T;W^{1,\infty}(0, L))}\}\leq K.
\end{align*}
\end{theorem}

\begin{remark}\label{rk:BKadmissible}
    The boundary and interface conditions defining $\mathfrak{P}_{\leq m}^{adm}$ ensure that any $(y_0, \vec{h})$ compatible with with respect to $\mu$ is automatically compatible with respect to any $\nu\in \mathfrak{P}_{\leq m}^{adm}$, so $z=z[\nu]$ is well-defined. This class is not sharp: the compatibility conditions constrain the difference $\mu-\nu$ rather than on $\mu$ and $\nu$ individually, so for a reference potential $\mu$, this set could be enlarged.
\end{remark}

\begin{remark}
    Along the proof, we apply the Carleman estimate for $(\omega_0, p)$ satisfying Hypothesis \ref{assumM}, where $\omega_0\Subset\omega$ is symmetric with respect to $L/2$. The symmetry and monotonicity hypothesis imposed on $p$ is then compatible with Hypothesis \ref{assumG}.
\end{remark}

\subsection{Some comments on the literature} The Korteweg-de Vries equation is one of the most celebrated nonlinear dispersive equations. The study of its controllability properties began with the early works of Russel and Zhang \cite{RZ93, RZ96}. Since then, extensive research has been conducted on its controllability properties. A good survey of results up to 2014 is provided by Cerpa \cite{Cer14}. Here, we briefly highlight some key issues.

Regarding our setting, where we aim to study the control properties of the equation, the uncontrolled system \eqref{eq:NLKdV}-\eqref{eq:tc}, which notably has a piecewise constant dispersion coefficient, was first proposed by Crépeau \cite{Crepeau16}. In that article, the boundary exact controllability with a single control acting on the Neumann boundary condition is established by a multiplier technique, under certain (smallness) conditions involving the coefficient $p$, the time $T>0$ and the length $L>0$. To the best of the author's knowledge, this is the only controllability result for such an equation. The exact boundary controllability of the KdV equation is a delicate issue, as was already noticed by Rosier \cite{Ros97} when he established (by a perturbative approach) that the exact controllability with right Neumann control holds if the length $L$ does not belong to a certain set of critical lengths. Later, using more refined nonlinear methods, Coron and Crépeau \cite{CC04} showed the (local) exact controllability of the nonlinear system holds even when the length is critical. Since then, extensive research has been conducted on various control problems surrounding this issue; a good survey of this phenomenon is given by Capistrano-Filho \cite{CF24}. As pointed out by Crépeau, obtaining exact control properties in the discontinuous setting, even by a perturbative approach, appears to be a challenging problem.

A less demanding property is the control to the trajectories. This was studied by Glass and Guerrero \cite{GG08} in the case of boundary controls. Observe that when $\by=0$, the control to the trajectories is known as null controllability. Thus, the control to the trajectories can be seen as a result in between the null controllability and exact controllability of the system. Moreover, it has been used as a stepping stone to obtain exact control properties by introducing additional controls acting on the system. For instance, Chapouly \cite{Cha09} employed this approach, leveraging control properties of the viscous Burgers equation.

Furthermore, the control properties of linear KdV equation in different settings have already been addressed by some authors. On one hand, the uniform controllability in the vanishing dispersion limit has been addressed by Glass and Guerrero \cite{GG08, GG09} under different boundary conditions. On the other hand, the controllability properties of the KdV equation in networks with various configurations has caught significant attention in recent years; see, for example, Capistrano-Filho, Parada and da Silva \cite{CFPS25} and the references therein.

With regard to global Carleman estimates for the KdV equation, when a variable main coefficient is considered Baudouin, Cerpa, Crépeau and Mercado \cite{BCCM14} derived a Carleman estimate for a sufficiently regular main coefficient, leading to Lipschitz stability in the inverse problem of retrieving the main coefficient of the equation. Here, we extend this result in a certain sense by allowing discontinuities in the main coefficient. 

Aiming to obtain controllability results and Lipschitz stability for certain inverse problems, global Carleman estimates for PDEs with discontinuous principal coefficients have been derived in various contexts. Most of these, however, impose either a monotonicity condition on the jump of the principal coefficient, strong geometric assumptions at the interface, or both. Given the extensive literature, we mention the early works addressing the heat equation by Doubova, Osses, and Puel \cite{DOP02}, the wave equation by Baudouin, Mercado, and Osses \cite{BMO07}, and the Schrödinger equation by Baudouin and Mercado \cite{BM08}. We also note that the nonphysical monotonicity condition on the principal coefficient has been relaxed in the parabolic setting, see Benabdallah, Dermenjian, and Le Rousseau \cite{BDLR07} in the one-dimensional case, while Le Rousseau and Robbiano \cite{LeRLR22} establish global (and local) Carleman estimates in the multi-dimensional case.

We emphasize that local Carleman estimates are a key tool in establishing unique continuation, observability, and controllability results. Moreover, their derivation typically requires less restrictive geometric conditions, in sharp contrast to the global case. For example, without any monotonicity assumptions or geometric conditions on the interface, local Carleman estimates were studied by Léautaud, Le Rousseau, and Robbiano \cite{LeRLR13} in the multi-dimensional parabolic case with applications to null controllability, and by Filippas \cite{Fil24} for the multi-dimensional wave equation with applications to quantitative approximate control. We also mention Imba \cite{Imb25} for the one dimensional wave equation, where a global Carleman estimate is derived with application to an inverse problem.

\subsection{Outline of the paper} The rest of the article is organized as follows. In \cref{S2} we establish the well-posedness of the uncontrolled KdV equation \eqref{eq:NLKdV-control}-\eqref{eq:tc} along with some regularity results for the linear KdV as well as for its adjoint. In \cref{S3} we derive a new two parameter Carleman estimate. In \cref{S4} we prove the control to the trajectories of the KdV equation. In \cref{S5} we prove the Lipschitz stability of the inverse problem of retrieving a potential term. In \cref{S6} we give final remarks about boundary observability.

\subsubsection{Notation} To make our computations clearer, the symbol $[\cdot]_a$ will denote the jump at $a\in \Gamma$, namely, for a function $\mu$ we write $[\mu]_a:=\mu(a^+)-\mu(a^-)$. Given two quantities $X$ and $Y$, we will employ the notation $X\lesssim Y$ to say that $X\leq CY$ for some $C>0$, possibly depending on several parameters involved in the computations. Often, we will use such notation when the constant does not matter on the analysis or when its dependency is understood.

\subsection*{Acknowledgments} Part of this article began while I was completing my master’s degree. I would like to warmly thank Nicolás Carreño and Alberto Mercado for introducing me into the world of research in control theory and for all their support during that period. 
I am also grateful to the anonymous referees for their valuable comments and suggestions.

This project has received funding from the European Union's Horizon 2020 research and innovation programme under the Marie Sk\l{}odowska-Curie grant agreement No 945332. Part of this work was also supported by the ANID-PFCHA/Magíster Nacional scholarship programme (2020-22201136).


\section{Well-posedness results}\label{S2}

Let us first recall the definition \eqref{S1:def:H3gamma} of $H_{\pw}^3(0, L)$ for $s=3$. By Sobolev embedding, if $u_{|_{I_k}}\in H^3(I_k)$, it also belongs to $C^{2}(\overline{I_k})$ and the operator
\begin{align}\label{S2:trace-op-cont}
    u_{|_{I_k}}\in H^3(I_k)\longmapsto u_{|_{I_k}}\in C^{2}(\overline{I_k})
\end{align}
is continuous. In particular, any $u\in H_{\pw}^3(0, L)$ satisfies $u_{|_{I_k}}\in C^{2}(\overline{I_k})$ for each $k\in\ran{1}{N-1}$. For $u\in H_{\pw}^3(0, L)$, we introduce the transmission conditions
\begin{align}\label{eq:tc-dom}
    \left\{\begin{array}{rrlll}
    u(a_k^-)&=&u(a_k^+), & ~k\in\ran{1}{N-1},\\
\sqrt{p_{k-1}}u'(a_k^-)&=&\sqrt{p_k}u'(a_k^+), & ~k\in\ran{1}{N-1},\\
p_{k-1}u''(a_k^-)&=&p_ku''(a_k^+), & ~k\in\ran{1}{N-1}.
    \end{array}\right.
\end{align}
The natural space for the study of the system \eqref{eq:NLKdV}-\eqref{eq:tc} is defined as
\begin{align*}
    \H_{\tc}^3(0, L)=\big\{u\in H_{\pw}^3(0, L)\ |\ u\text{ satisfies } \eqref{eq:tc-dom}\big\},
\end{align*}
which is a closed subspace of $H_{\pw}^3$ and therefore a Hilbert space endowed with the inherited inner product of $H_{\pw}^3$. Additionally, let us introduce the Banach space
\begin{align*}
    \X_T^0(0, L)=C([0, T], L^2(0, L))\cap L^2(0, T; H^1(0, L)),
\end{align*}
equipped with the norm
\begin{align*}
    \norm{\cdot}_{\X_T^0(0, L)}=\norm{\cdot}_{C^0([0, T]; L^2(0, L))}+\norm{\cdot}_{L^2(0, T; H^1(0, L))}.
\end{align*}

\subsection{Linear Cauchy problem}\label{S2:wp-linear} Let us consider the nonhomogeneous boundary-value problem
\begin{align}\label{eq:bvp-kdv-lin-f}
\left\{
\begin{array}{rll}
y_t+p(x)y_{xxx}=f, & ~(t, x)\in (0, T)\times (0, L),\\
y(t, 0)=h_1(t),\ y(t, L)=h_2(t),\ y_{x}(t, L)=h_3(t), & ~t\in (0, T),\\
y(0, x) =  y_0(x), & ~x\in (0, L),
\end{array}
\right.
\end{align}
coupled along with the transmission conditions \eqref{eq:tc}. Let $\A:\dom(\A)\subset L^2(0, L)\to L^2(0, L)$ be the formally defined linear operator given by $\A=-p(x)\partial_x^3$ with domain 
\begin{align*}
    \dom(\A):=\{u\in H_{\pw}^3(0, L)\ |\ u(0)=u(L)=u'(L)=0\ \text{ and } u \text{ satisfies } \eqref{eq:tc-dom}\}.
\end{align*}
Let us also introduce its formal adjoint operator $\A^*:=p(x)\partial_x^3$ with domain $\dom(\A^*)$ given by those functions $z\in H_{\pw}^3(0, L)$ satisfying $z(0)=z(L)=z_x(0)=0$ and the transmission conditions \eqref{eq:tc-dom}.

\begin{lemma}\label{lem:Awelldef}
    The operators $\A$ and $\A^*$ are well-defined.
\end{lemma}
\begin{proof}
    Let $z\in \D(\A)$ and set $g:=\sum_{k=0}^{N-1} p_k z_{xxx}(x)\mathbbm{1}_{I_k}$. As $z_{|_{I_k}}\in H^3(I_k)$ for all $k\in\ran{0}{N-1}$, we have $g_{|_{I_k}}=p_kz_{xxx}\in L^2(I_k)$ and thus $g\in L^2(0, L)$. If $\vp\in C_c^\infty(0, L)$, by performing integration by parts on each $I_k$ and adding up all these integrals, we get
    \begin{align*}
        \int_0^L g\vp dx=\sum_{k=0}^{N-1} \int_{a_k}^{a_{k+1}} pz_{xxx} \vp dx&=\sum_{k=1}^N \left(-\int_{a_{k-1}}^{a_k}p_{k-1}z_{xx}\vp_x dx+p_{k-1}z_{xx} \vp\big|_{a_{k-1}}^{a_k}\right)\\
        &=-\int_0^L pz_{xx}\vp_xdx+\sum_{k=1}^{N-1} \vp(a_k)[pz_{xx}]_{a_k}.
    \end{align*}
    The interface terms above vanish due to the transmission condition $[p z_{xx}]_{a_k}=0$ for $k\in\ran{1}{N-1}$. Since $\vp$ is arbitrary, the previous identity implies that $g$ is the weak derivative of $pz_{xx}$. That is, the weak derivative of $pz_{xx}$ coincides a.e. with the piecewise derivative $p_k z_{xxx}$ on each $I_k$. Since $g\in L^2(0, L)$, then $pz_{xx}\in H^1(0, L)$ and $\A z=-g\in L^2(0, L)$ is well-defined. The proof for $\A^*$ is identical.
\end{proof}

We can now employ semigroup theory tools to study the linear Cauchy problem \eqref{eq:kdv-lin-f}-\eqref{eq:tc}.

\begin{proposition}\label{prop:semigroup}
The operators $\A$ and $\A^*$ both generate a strongly continuous semigroup of contractions on $L^2(0, L)$.
\end{proposition}
\begin{proof}
The operators $\A$ and $\A^*$ are both closed and densely defined. If $z\in \D(\A)$ then
\begin{align*}
    \inn{\A z, z}_{L^2(0, L)}&=-\int_0^L p(x)z_{xxx}zdx\\
    &=\sum_{k=1}^N \left(\int_{a_{k-1}}^{a_k}p_{k-1}z_{xx}z_xdx-p_{k-1}z_{xx}z\Big|_{a_{k-1}}^{a_k}\right)\\
    &=\sum_{k=1}^N\left(\dfrac{p_{k-1}}{2} \big(|z_x(a_k^-)|^2-|z_x(a_{k-1}^+)|^2\big)-p_{k-1}z_{xx}(a_k^-)z(a_k^-)+p_{k-1}z_{xx}(a_{k-1}^+)z(a_{k-1}^+)\right)\\
    &=-\dfrac{p_0}{2}|z_x(0)|^2+\sum_{k=1}^{N-1}\left(-\dfrac{1}{2}[p|z_x|^2]_{a_k}+z(a_k)[pz_{xx}]_{a_k}\right)\\
    &=-\dfrac{p_0}{2}|z_x(0)|^2\leq 0.
\end{align*}
In a similar way, $\A^*$ is also dissipative since
\begin{align*}
    \inn{z, \A^* z}_{L^2(0, L)}=-\dfrac{p_{N-1}}{2}|z_x(L)|^2\leq 0.
\end{align*}
The conclusion follows from the classical Lumer-Phillips Theorem.
\end{proof}

By classical semigroup theory \cite[Chapter 4]{Pazy83}, the previous proposition implies that the linear problem \eqref{eq:bvp-kdv-lin-f}-\eqref{eq:tc} with homogeneous boundary data $\vec{h}=(h_1, h_2, h_3)=(0, 0, 0)$ admits a unique classical solution in the class $C^0([0, T], \D(\A))\cap C^1([0, T], L^2(0, L))$ whenever $f\in C^1([0, T], L^2(0, L))$ and $y_0\in\D(\A)$. To treat the nonhomogeneous linear problem,
for $T>0$ we set
\begin{align*}
    \mc{Z}_{0, T}:=L^2(0, L)\times H^1(0, T)\times H^1(0, T)\times L^2(0, T).
\end{align*}
Using semigroup theory and the multipliers method, we obtain the following well-posedness result for the linear Cauchy problem \eqref{eq:bvp-kdv-lin-f}-\eqref{eq:tc} along with the classical Kato smoothing effect.

\begin{proposition}\label{prop:wp-kdv-f-1}
Let $T>0$. Let $f\in L^1(0, T; L^2(0, L))$ and $(y_0, \Vec{h})\in \mc{Z}_{0, T}$. Then there exists a unique solution $y$ of the KdV equation \eqref{eq:bvp-kdv-lin-f}-\eqref{eq:tc} that belongs to $\X_T^0(0, L)$. Also, there exists $C=C(T, L, \Gamma, p)>0$ such that
\begin{align}\label{ineq:wp-kdv-f-1}
    \norm{y}_{\X_T^0(0, L)}\leq C\big(\norm{(y_0, \Vec{h})}_{\mc{Z}_{0, T}}+\norm{f}_{L^1(0, T; L^2(0, L))}\big).
\end{align}
\end{proposition}

\begin{proof}
Let us assume that $f\in C^1([0, T], L^2(0, L))$, $y_0\in \H_{\tc}^3(0, L)$ and $\Vec{h}=(h_1, h_2, h_3)\in C^\infty([0, T])^3$ satisfying the compatibility condition
\begin{align*}
    y_0(0)=h_1(0),\ y_0(L)=h_2(0)\ \text{ and }\ y_0'(L)=h_3(0).
\end{align*}
Set $\ld:=4\sum_{k=1}^{N-1}(L-a_k)^{-1}$ and let $\psi^D$ be defined to lift the Dirichlet boundary data,
\begin{align*}
    \psi^D(t, x):=\dfrac{(L-x)^2}{L^2}\prod_{k=1}^{N-1}\left(\dfrac{x-a_k}{a_k}\right)^4h_1(t)+\left(\frac{x(2L-x)}{L^2}+\frac{\ld x(L-x)}{L}\right)\prod_{k=1}^{N-1}\left(\dfrac{x-a_k}{L-a_k}\right)^4h_2(t),
\end{align*}
which satisfies $\psi^D(t, 0)=h_1(t)$, $\psi^D(t, L)=h_2(t)$, $\psi^D_x(t, L)=0$ and the transmission conditions \eqref{eq:tc}with $[\psi^D]_{a_k}=[\sqrt{p}\psi_x^D]_{a_k}=[p\psi_{xx}^D]_{a_k}=0$ for each $k\in\ran{1}{N-1}$. Thereby, $z:=y-\psi^D$ satisfies
\begin{align}\label{eq:kdv-lift}
\left\{
\begin{array}{rll}
z_t+p(x)z_{xxx}=\widetilde{f}, & ~(t, x)\in (0, T)\times (0, L),\\
z(t, 0)=z(t, L)=0,\ z_{x}(t, L)=h_3(t), & ~t\in (0, T),\\
z(0, x)= y_0(x)-\psi^D(0, x), & ~x\in (0, L),
\end{array}
\right.
\end{align}
coupled with the transmission conditions \eqref{eq:tc} and $\widetilde{f}:=f-(\psi^D_t+p(x)\psi^D_{xxx})\in C^1([0, T], L^2(0, L))$. By means of semigroup theory \cite[Chapter 4]{Pazy83} to treat the nonhomogeneous term $\widetilde{f}$, we can once again choose $\psi^N$ to lift the boundary Neumann data \cite[Proposition 2]{Crepeau16} so that $w:=z-\psi^N$ is a classical solution in $C^0([0, T], \D(\A))\cap C^1([0, T], L^2(0, L))$ of the corresponding homogeneous boundary problem with source term $g:=\widetilde{f}-(\psi_t^N+p(x)\psi_{xxx}^N)\in C^1([0, T], L^2(0, L))$ and initial data $w(0, \cdot)=y_0-\psi^D(0, \cdot)-\psi^N(0, \cdot)\in\D(\A)$. From this we obtain a unique classical solution solution $z=w+\psi^N\in C^0([0, T], \H_{\tc}^3(0, L))\cap C^1([0, T], L^2(0, L))$ to \eqref{eq:kdv-lift}.

Let $q\in C^0([0, T]\times [0, L])$ be such that $q_{|_{I_k}}\in C^\infty([0, T]\times \overline{I_k})$ for $k\in\ran{0}{N-1}$. Performing several integration by parts we obtain
\begin{align*}
    2\int_0^s\int_{I_k} qzz_tdxdt=-\int_0^s\int_{I_k} q_t|z|^2dxdt+\int_{I_k}q|z|^2\bigg|_0^sdx
\end{align*}
and
\begin{multline*}
    2\int_0^s\int_{I_k}qpzz_{xxx}dxdt=-\int_0^s\int_{I_k}pq_{xxx}|z|^2dxdt+3\int_0^s\int_{I_k} pq_x|z_x|^2dxdt\\
    +\int_0^s(pq_{xx}|z|^2-pq|z_x|^2-2pq_xzz_x+2pqzz_{xx})\bigg|_{a_k}^{a_{k+1}}dt,
\end{multline*}
for each $k\in\ran{0}{N-1}$. By adding these equations we get
\begin{multline*}
    3\int_0^s\int_0^L pq_x|z_x|^2dxdt+\int_0^L q|z|^2\bigg|_0^sdx=\int_0^s\int_0^L (q_t+pq_{xxx})|z|^2dxdt+2\int_0^s\int_0^L qz\widetilde{f}dxdt\\
    +\int_0^s p_{N-1}q(L)|h_3(t)|^2dt-\int_0^s p_0q(0)|z_x(t, 0)|^2dt+\sum_{k=1}^{N-1}\int_0^s\big([pq_{xx}]_{a_k}|z(t, a_k)|^2-[q]_{a_k}p_k|z_x(t, a_k^+)|^2\big)dt\\+2\sum_{k=1}^{N-1}\int_0^s\big([q]_{a_k} z(t, a_k)p_kz_{xx}(t, a_k^+)-[\sqrt{p}q_x]_{a_k}z(t, a_k)\sqrt{p_k}z_x(t, a_k^+)\big)dt.
\end{multline*}
On the one hand, if we set $q\equiv 1$ on $[0, L ]$, we obtain
\begin{align*}
    \int_0^s p_0|z_x(t, 0)|^2dt+\int_0^L |z(s, x)|^2dx=\int_0^L |z(0, x)|^2dx+\int_0^sp_{N-1}q(L)|h_3(t)|^2dt+2\int_0^s\int_0^L z\widetilde{f}dxdt
\end{align*}
By Cauchy-Schwarz and then Young's inequality, we get
\begin{align*}
    \norm{z}_{C^0([0, T], L^2(0, L))}^2\leq C\big(\norm{z(0,\cdot)}_{L^2(0, L)}^2+\norm{h_3}_{L^2(0, T)}^2+\norm{\widetilde{f}}_{L^1(0, T; L^2(0, L))}^2\big).
\end{align*}
On the other hand, by setting $s=T$ and choosing $q_0(x)=x/\sqrt{p_0}$ for $x\in I_0$ and $q_k(x)=(x-a_k)/\sqrt{p_k}+q_{k-1}(a_k^-)$ for $x\in I_k$ for all $k\in\ran{1}{N-1}$, we readily get the identity
\begin{multline*}
    3\int_0^T\int_0^L \sqrt{p}|z_x|^2dxdt+\int_0^L q|z(T, x)|^2dx=\int_0^L q|z(0, x)|^2dx\\+\int_0^Tp_{N-1}q(L)|h_3(t)|^2dt+2\int_0^T\int_0^L qz\widetilde{f}dxdt.
\end{multline*}
The above identity together with the previous estimates implies
\begin{align*}
    \norm{z_x}_{L^2(0, T; L^2(0, L))}^2\leq C\big(\norm{z(0,\cdot)}_{L^2(0, L)}^2+\norm{h_3}_{L^2(0, T)}^2+\norm{\widetilde{f}}_{L^1(0, T; L^2(0, L))}^2\big).
\end{align*}
First, using that $H^1(0, T)\hookrightarrow C^0([0, T])$, we have
\begin{align*}
    \norm{z(0, \cdot)}_{L^2(0, L)}\leq C\big(\norm{y_0}_{L^2(0, L)}+\norm{(h_1, h_2)}_{H^1(0, T)^2}\big).
\end{align*}
Second, recalling that $\widetilde{f}=f-(\psi^D_t+p(x)\psi^D_{xxx})$, we get
\begin{align*}
    \norm{\widetilde{f}}_{L^1(0, T; L^2(0, L))}\leq C\big(\norm{(h_1, h_2)}_{H^1(0, T)^2}+\norm{f}_{L^1(0, T; L^2(0, L))}\big).
\end{align*}
Third, we have
\begin{align*}
    \norm{\psi^D}_{\mc{X}_T^0(0, L)}\leq C\big(\norm{(h_1, h_2)}_{C^0([0, T])^2}+\norm{(h_1, h_2)}_{L^2(0, T)^2}\big)\leq C\norm{(h_1, h_2)}_{H^1(0, T)^2}.
\end{align*}
Gathering the above estimates and using that $y=z+\psi^D$, we obtain estimate \eqref{ineq:wp-kdv-f-1}. Finally, for data $f\in L^1(0, T; L^2(0, L))$ and $(y_0, \Vec{h})\in \mc{Z}_{0, T}$, linearity and a classical density argument through smooth enough compatible data yields the desired unique solution $y\in \mc{X}_T^0(0, L)$; see for instance Coron \cite[Section 2.2]{Cor07}.
\end{proof}

\begin{remark}
The multiplier employed in the previous proof was introduced by Crépeau \cite{Crepeau16}, capturing the transmission conditions and being similar in spirit to the one used by Rosier \cite{Ros97}.
\end{remark}

\subsection{Nonlinear system} Let $\mu\in L^\infty(0, L)$ and $z\in \X_T^0(0, L)$, and consider the system
\begin{align}\label{eq:NLKdV-fixedpoint}
\left\{
\begin{array}{rll}
y_t+p(x)y_{xxx}=-\mu(x)z_x-zz_x, & ~(t, x)\in (0, T)\times (0, L),\\
y(t, 0)=h_1(t),\ y(t, L)=h_2(t),\ y_{x}(t, L)=h_3(t), & ~t\in (0, T),\\
y(0, x) =  y_0(x), & ~x\in (0, L),
\end{array}
\right.
\end{align}
coupled by the transmission conditions \eqref{eq:tc}. Due to the Kato-smoothing effect and \cref{app:prop:yyxH1}, $-zz_x$ is allowed as a source term and for $(y_0, \Vec{h})\in \mc{Z}_{0, T}$ we can introduce the map $\F_{(y_0, \vec{h})}: \X_T^0(0, L)\to \X_T^0(0, L)$ to be the map defined by $\F_{(y_0, \Vec{h})}(z)=y$, where $y$ is the solution of \eqref{eq:NLKdV-fixedpoint}. 

We now establish global well-posedness using a classical fixed-point argument. Assuming more regularity on the initial data, we can guarantee the existence of a more regular solution following Bona-Sun-Zhang \cite[Theorem 4.1]{BSZ03}. To this end, let us introduce for $j\in \N$,
\begin{align*}
\begin{array}{c}
    \mc{Z}_{3j, T}:=H_{\pw}^{3j}(0, L)\times H^{1+j}(0, T)\times H^{1+j}(0, T)\times H^j(0, T),\\
    \X_{\pw, T}^{3j}(0, L)=C([0, T], H_{\pw}^{3j}(0, L))\cap L^2(0, T; H_{\pw}^{3j+1}(0, L)),
\end{array}
\end{align*}
both of them equipped with their natural norms. Since we need data that is compatible with the boundary and interface conditions of our equation, we introduce the following definition.
\begin{definition}\label{def:compatible}
    Let $\mu\in H_{\pw}^{3(j-1)}(0, L)$. Given $(y_0, \Vec{h})\in \mc{Z}_{3j, T}$ define $(\phi_\ell)_{\ell\geq 1}$ recursively by
    \begin{align*}
        \left\{\begin{array}{lc}
            \phi_0:=y_0, & \\
            \phi_\ell:=-p(x)(\phi_{\ell-1})'''-\mu(x)\phi_{\ell-1}'-\sum_{i=0}^{\ell-1}(\phi_{i}\phi_{\ell-1-i})'     & \text{on each } I_k \text{ for } k\in\ran{0}{N-1}.
        \end{array}\right.
    \end{align*}
    We say that $(y_0, \Vec{h})\in \mc{Z}_{3j, T}$ is $3j$-compatible if
    \begin{align*}
    \left\{\begin{array}{ll}
        \phi_\ell(0)=h_1^{(\ell)}(0),\ \phi_\ell(L)=h_2^{(\ell)}(0),\ \phi_\ell'(L)=h_3^{(\ell)}(0),     & \ell\in\ran{0}{j-1}, \\
        \phi_\ell \text{ satisfies the transmission conditions } \eqref{eq:tc-dom}, & \ell\in\ran{0}{j-1}.
    \end{array}\right.
    \end{align*}
\end{definition}

The result is the following.

\begin{proposition}\label{prop:NLwp-Xs}
Let $T>0$, $L>0$ and $\mu\in L^\infty(0, L)$. For any $(y_0,\Vec{h})\in \mc{Z}_{0, T}$ the system \eqref{eq:NLKdV}-\eqref{eq:tc} has a unique solution $y\in \X_T^0(0, L)$ satisfying
\begin{align}\label{S1:prop:ineq:KdV-X0}
    \norm{y}_{\X_T^0(0, L)}\leq C\big(\norm{(y_0, \vec{h})}_{\mc{Z}_{0, T}}\big)\norm{(y_0, \vec{h})}_{\mc{Z}_{0, T}},
\end{align}
for some continuous, non-decreasing $C: \R^+\to\R^+$ depending on $\norm{\mu}_{L^\infty(0, L)}$. If, additionally, $j\geq 1$, $\mu\in H_{\pw}^{3(j-1)}(0, L)$ and $(y_0, \vec{h})\in \mc{Z}_{3j, T}$ is $3j$-compatible data, then $y\in \X_{\pw, T}^{3j}(0, L)$ with continuous dependency on $(y_0, \vec{h})$.
\end{proposition}
\begin{proof}
Observe that the constant $C>0$ given by \cref{prop:wp-kdv-f-1} is affine on $T>0$. Let us set $\mc{F}=\mc{F}_{(y_0, \vec{h})}$. Then, for any $0<\tau\leq T$, we have
\begin{align*}
    \norm{\F(z)}_{\X_\tau^0(0, L)}\leq C\left(\norm{(y_0, \vec{h})}_{\mc{Z}_{0, T}}+\norm{\mu z_x}_{L^1(0, \tau; L^2(0, L)}+\norm{zz_x}_{L^1(0, \tau;L^2(0, L))}\right).
\end{align*}
By using \cref{app:lem:L1Hs-estimate} on $(0, L)$, the previous estimate implies
\begin{align*}
    \norm{\F(z)}_{\X_\tau^0(0, L)}\leq C\norm{(y_0, \vec{h})}_{\mc{Z}_{0, \tau}}+C_1\norm{\mu}_{L^\infty}\tau^{1/2}\norm{z}_{\X_\tau^0(0, L)}+C_1(\tau^{1/2}+\tau^{1/3})\norm{z}_{\X_\tau^0(0, L)}^2.
\end{align*}
for some $C_1>0$. Let $R>0$ be such that
\begin{align*}
    R=2C\norm{(y_0, \vec{h})}_{\mc{Z}_{0, T}}\ \text{ and }\ C_1\norm{\mu}_{L^\infty}\tau^{1/2}+C_1(\tau^{1/2}+\tau^{1/3})R \leq \dfrac{1}{4}.
\end{align*}
Then for small $\tau$, the map $\F$ reproduces closed ball $\mathbb{B}_R=\{z\in \X_\tau^0: \norm{z}_{\X_\tau^0(0, L)}\leq R\}$ and
\begin{align*}
    \norm{\F(z_1)-\F(z_2)}_{\X_\tau^0}&\leq \big(C_1\norm{\mu}_{L^\infty}\tau^{1/2}+C_1(\tau^{1/2}+\tau^{1/3})(\norm{z_1}_{\X_\tau^0}+\norm{z_2}_{\X_\tau^0})\Big)\norm{z_1-z_2}_{\X_\tau^0}\\
    &\leq \dfrac{1}{2}\norm{z_1-z_2}_{\X_\tau^0}.
\end{align*}
The smallness condition imposed by $R$ allows us to apply the Banach fixed-point theorem in time $\tau\leq T$, which give us a unique fixed point $y$ of $\F$ belonging to $\mathbb{B}_R$ and by consequence being the unique solution $y\in \X_\tau^0(0, L)$ of \eqref{eq:NLKdV-traj}. To globalize the solution, we multiply the nonlinear equation by $y$ (that is, choosing $q\equiv 1$ in the analogous multiplier identity obtained in \cref{prop:wp-kdv-f-1}) and then integrate by parts in $(0, L)$. First, we manage the potential contribution $\big|\int_0^L \mu(x)yy_xdx\big|\leq\norm{\mu}_{L^\infty}\norm{y}_{L^2}\norm{y_x}_{L^2}$, and then noticing that the nonlinear term contributes $\int_0^L y(yy_x)dx=\left.\frac{1}{3}y^3\right|_{0}^L=\frac{1}{3}(h_2^3-h_1^3)$ which are controlled by Sobolev embedding $H^1(0, T)\hookrightarrow C^0([0, T])$, by a Grönwall estimate we get
\begin{align*}
    \sup_{t\in [0, T]} \norm{y(t)}_{L^2(0, L)}\leq C(\norm{(y_0, \vec{h})}_{\mc{Z}_{0, T}})\norm{(y_0, \vec{h})}_{\mc{Z}_{0, T}}.
\end{align*}

Hence we can pick $\tau\in (0, T]$ only depending on $\norm{(y_0, \vec{h})}_{\mc{Z}_{0, T}}$ and, up to shrinking $\tau$ so that $n\tau=T$ for some $n\in \N$, we can extend the previous argument on intervals $(\tau, 2\tau], (2\tau, 3\tau],\ldots, ((n-1)\tau, n\tau=T]$. Along with an standard uniqueness argument, the existence of a unique solution $y\in \X_T^0(0, L)$ is guaranteed. A bootstrap argument lead us to estimate \eqref{S1:prop:ineq:KdV-X0}.

If $(y_0, \vec{h})\in \mc{Z}_{3, T}$, we can set $z:=y_t$ and look at the equation it satisfies. We can replicate the previous reasoning to get that $\norm{z}_{\X_T^0(0, L)}\leq C_1\norm{(y_0, \vec{h})}_{\mc{Z}_{3, T}}$, where $C_1$ is continuous non-decreasing in $\norm{(y_0, \vec{h})}_{\mc{Z}_{0, T}}$. Thus, by carefully using the fact that $z=-p(x)y_{xxx}-\mu(x) y_x-yy_x$ (recall that $p$ is piecewise constant) we obtain that $y\in \mc{X}_{\pw, T}^3(0, L)$. Following a similar procedure one can establish the existence of solutions in $\mc{X}_{\pw, T}^{3j}(0, L)$ for $j\ge 2$, whenever $\mu\in H_{\pw}^{3(j-1)}(0, L)$ to handle the potential term. The details are easily fulfilled following \cite[Section 4]{BSZ03}.
\end{proof}

\begin{remark}\label{rk:bdry-reg}
    In view of the results in the smooth case \cite[Theorem 1.3]{BSZ03}, it is most likely that the functional space $\mc{Z}_{3j, T}$ is not sharp,. Obtaining the corresponding smoothing effects to lower the regularity of the boundary data $(h_1, h_2)$ in our discontinuous setting is an interesting problem.
\end{remark}

\subsection{Adjoint system} We now introduce the notion of weak solutions that will be used for our controllability problem with homogeneous boundary conditions in \cref{S3}. Let us consider
\begin{align}\label{eq:kdv-lin-f}
\left\{
\begin{array}{rll}
y_t+p(x)y_{xxx}+y_x=f, & ~(t, x)\in (0, T)\times (0, L),\\
y(t, 0)=y(t, L)=y_{x}(t, L)=0, & ~t\in (0, T),\\
y(0, x) =  y_0(x), & ~x\in (0, L),
\end{array}
\right.
\end{align}
coupled by the corresponding transmission conditions \eqref{eq:tc}. The following definition is motivated by performing integration by parts as the ones done in \cref{prop:semigroup}.

\begin{definition}\label{def:weak-sol}
    For $(f, y_0)\in L^2(0, T; H^{-1}(0, L))\times L^2(0, L)$ a function $y\in C([0, T], L^2(0, L))$ is called a weak solution of \eqref{eq:kdv-lin-f}-\eqref{eq:tc} if it satisfies
    \begin{align*}
        \iint_Q ygdxdt+(y(T), \vp_T)_{L^2(0, L)}=\int_0^T \inn{f, \vp}_{H^{-1}(0, L)\times H_0^1(0, L)}dt+(y_0,\vp(0))_{L^2(0, L)},
    \end{align*}
    for all $(g, \vp_T)\in L^1(0, T; L^2(0, L))\times L^2(0, L)$, where $\vp$ is the mild solution of the adjoint system
    \begin{align}\label{eq:kdvadj}
        \left\{
        \begin{array}{rll}
        -\vp_t-p(x)\vp_{xxx}-\vp_x=g, & ~(t, x)\in (0, T)\times (0, L),\\
        \vp(t, 0)=\vp(t, L)=\vp_{x}(t, 0)=0, & ~t\in (0, T),\\
        \vp(T, x) =  \vp_T(x), & ~x\in (0, L),
        \end{array}
        \right.
    \end{align}
    coupled by the corresponding transmission conditions \eqref{eq:tc}.
\end{definition}

Let $\widetilde{\A}:=-p(x)\partial_x^3-\partial_x$ with $\D(\widetilde{\A})=\D(\A)$. Straightforward computations show that the conclusions of \cref{prop:wp-kdv-f-1} still hold true when $\A$ is replaced by $\widetilde{\A}$. Hence, following the same approach as in \cref{prop:wp-kdv-f-1} but for the adjoint equation, we can ensure that for any $\vp_T\in L^2(0, L)$, there is a unique mild solution $\vp$ of \eqref{eq:kdvadj} which belongs to $\X_T^0(0, L)$. In particular, it also enjoys the Kato-type smoothing effect and henceforth the above definition makes sense. From now on we relabel $\widetilde{\A}$ by $\A$.

\begin{remark}
    A simple computation using integration by parts shows that a (piecewise) regular solution $y$ of \eqref{eq:kdv-lin-f}-\eqref{eq:tc} is also a solution in the above sense.
\end{remark}

We now establish some regularity estimates for the adjoint system \eqref{eq:kdvadj} that will be needed later in \cref{S4}. We recall the definition of $H_{\pw}^s(0, L)$ given in \eqref{S1:def:H3gamma} and we remark that in the next result we make the notational convention $H_{\pw}^{-1}(0, L):=H^{-1}(0, L)$, where $H^{-1}$ is the usual dual space of $H_0^1$ equipped with the dual norm
\begin{align*}
    \norm{\psi}_{H^{-1}}=\sup_{\substack{h\in H_0^1(0, L)\\ \norm{h}_{H_0^1}\leq 1}}\left|\int_0^L \psi hdx\right|,
\end{align*}
as a consequence of Riesz's representation theorem.

\begin{proposition}\label{prop:wp-kdvadj-Xs}
Let $T>0$. If $\vp_T\in \D(\A^*)$ and $g\in L^1(0, T; \D(\A^*))$, then there exists a unique strong solution $\vp$ of the adjoint equation \eqref{eq:kdvadj}-\eqref{eq:tc} such that, for some $C=C(T, L, \Gamma, p)>0$,
\begin{align}\label{prop:eq:kdvadj-l1Hs}
    \norm{\vp}_{C([0, T], H_{\pw}^3(0, L))\cap L^2(0, T; H_{\pw}^4(0, L))}\leq C\big(\norm{\vp_T}_{H_{\pw}^3(0, L)}+\norm{g}_{L^1(0, T; H_{\pw}^3(0, L))}\big).
\end{align}
Additionally, if $g\in L^2(0, T; \D(\A^*))$, then for $s\in \{0, 1, 2, 3\}$, we have
\begin{align}\label{prop:eq:kdvadj-l2Hs}
    \norm{\vp}_{L^2(0, T; H_{\pw}^{s+1}(0, L))}\leq C\big(\norm{\vp_T}_{H_{\pw}^s(0, T)}+\norm{g}_{L^2(0, T; H_{\pw}^{s-1}(0, L))}\big).
\end{align}
\end{proposition}
\begin{proof}
    The existence follows by classical semigroup theory as done in \cref{prop:wp-kdv-f-1}, see \cite[Chapter 4]{Pazy83}. We thus focus on obtaining estimates \eqref{prop:eq:kdvadj-l2Hs}.\\

    \noindent\emph{Step 1: case $s=0$.} This case is handled exactly as in \cref{prop:wp-kdv-f-1}, but instead, we choose the multiplier $q_{N-1}(x)=(x-L)/\sqrt{p_{N-1}}$ for $x\in I_{N-1}$ and $q_k(x)=(x-a_{k+1})/\sqrt{p_{k}}+q_{k+1}(a_{k+1}^+)$ for $x\in I_k$ for all $k\in\ran{0}{N-2}$, we obtain
    \begin{multline*}
        3\int_0^T\int_0^L \sqrt{p}|\vp_x|^2dxdt+\int_0^L |q\vp(0, x)|^2dx=\int_0^L |q||\vp_T(x)|^2dx\\+\int_0^T\int_0^L\frac{1}{\sqrt{p}}|\vp|^2dxdt+2\int_0^T\int_0^L |q|\vp gdxdt.
    \end{multline*}
    Thus, since $p$ is bounded from below, by using Poincaré's inequality, for any $\veps>0$ we have
    \begin{align*}
        \norm{\vp_x}_{L^2(0, T; L^2(0, L))}^2\leq C\norm{\vp_T}_{L^2(0, L)}^2+\veps\norm{\vp_x}_{L^2(0, T; L^2(0, L))}^2+C_\veps\norm{g}_{L^2(0,T; H^{-1})}^2.
    \end{align*}
    By choosing $\veps>0$ small enough, we readily get inequality \eqref{prop:eq:kdvadj-l2Hs} for $s=0$.\\

    \noindent\emph{Step 2: case $s=3$.} Let $\vp_T\in \D(\A^*)$ and $g\in L^1(0, T; \D(\A^*))$. By classical semigroup theory, then $\vp\in C([0, T], \D(\A^*))$. Hence, if we let $w=\A^*\vp$, it is a mild solution of \eqref{eq:kdvadj} with initial data $\A^*\vp_T$ and source term $\A^* g\in L^1(0, T; L^2(0, L))$. Furthermore, we can perform the same analysis as in the case $s=0$, which lead us to
    \begin{align}\label{ineq:s3ineq1}
        \norm{w}_{L^2(0, T; H^1(0, L))}\leq C\big(\norm{\A^* \vp_T}_{L^2(0, L)}^2+\norm{\A^* g}_{L^2(0, T; H^{-1})}^2\big).
    \end{align}
    To estimate the last term in the right-hand side above, we use $\A^*g=pg_{xxx}+g_x$ and
    \begin{align*}
        \norm{\A^* g}_{H^{-1}}=\sup_{\substack{h\in H_0^1(0, L)\\ \norm{h}_{H_0^1}\leq 1}}\left|\int_0^L \big(pg_{xxx}+g_x\big) h dx\right|.
    \end{align*}
    As in \cref{lem:Awelldef}, due to the transmission conditions we have $\A^*g\in L^2(0, L)$ and $\inn{\A^* g, h}_{L^2(0, L)}=-\inn{pg_{xx}+g, h_x}_{L^2(0, L)}$.
    By Cauchy-Schwarz's inequality, we get
    \begin{align}\label{ineq:s3ineq2}
        \norm{\A^* g}_{H^{-1}}=\sup_{\substack{h\in H_0^1(0, L)\\ \norm{h}_{H_0^1}\leq 1}}\left|\int_0^L \big(pg_{xx}+g\big) h_x dx\right|\leq \norm{pg_{xx}}_{L^2(0, L)}+\norm{g}_{L^2(0, L)}\leq C\norm{g}_{H_{\pw}^2(0, L)}.
    \end{align}
    As $p$ is bounded from below, by using the transmission conditions again, we have that
    \begin{align}\label{ineq:s3ineq3}
        \norm{\vp}_{L^2(0, T;H_{\pw}^4(0, L))}\lesssim \norm{w}_{L^2(0, T; H^1(0, L))}+\norm{\vp}_{L^2(0, T; H_{\pw}^2(0, L))},
    \end{align}
    and gathering inequalities \eqref{ineq:s3ineq1}-\eqref{ineq:s3ineq2}-\eqref{ineq:s3ineq3}, we have
    \begin{align*}
        \norm{\vp}_{L^2(0, T;H_{\pw}^4(0, L))}\lesssim \norm{\vp_T}_{H_{\pw}^3(0, L)}+\norm{g}_{L^2(0, T; H_{\pw}^2(0, L))},
    \end{align*}
    proving estimate \eqref{prop:eq:kdvadj-l2Hs} for $s=3$.\\

    \noindent\emph{Step 3: cases $s=1, 2$.} Let us denote by $\mathcal{S}$ the (linear) solution map that sends $(\vp_T, g)$ into $\vp$. From the previous steps, such map is continuous in the following spaces,
    \begin{align*}
    \begin{array}{cccc}
        \mathcal{S}: & L^2(0, L)\times L^2(0, T; H^{-1}(0, L)) & \longrightarrow & L^2(0, T; H_{\pw}^{1}(0, L)),\\
        \mathcal{S}: &\D(\A^*)\times L^2(0, T; \H_{tc,\mathrm{DDN}}^{2}(0, L)) & \longrightarrow & L^2(0, T; H_{\pw}^{4}(0, L)),
    \end{array}
    \end{align*}
    where $\H_{\tc,\mathrm{DDN}}^{2}(0,L)$ consists of functions in $H_{\pw}^2(0, L)$ encoding the transmission conditions up to order $1$, the Neumann at $0$, and Dirichlet boundary conditions. By linear interpolation, \cref{app:prop:interpolation-dom}-\cref{rk:interpolationBC} and \cref{app:lem:interpolation-source}, estimate \eqref{prop:eq:kdvadj-l2Hs} follows.
\end{proof}

We now establish a similar result for the adjont system to the linearized version of \eqref{eq:NLKdV-control} with a regular source term. These estimates are key for the proof of \cref{S4:prop:carleman-I}.

\begin{proposition}\label{prop:wp:adj-linearized}
    Let $T>0$ be given and assume $\by\in \X_T^0(0, L)$. Then for any $\vp_T\in L^2(0, L)$ and $g\in L^1(0, T; L^2(0, L))$, there exists a unique solution $\vp\in \X_T^0(0, L)$ of
    \begin{align*}
        \left\{
        \begin{array}{rll}
        -\vp_t-p(x)\vp_{xxx}-\vp_x-\by \vp_x = g, & ~(t, x)\in (0, T)\times (0, L),\\
        \vp(t, 0)=\vp(t, L)=\vp_{x}(t, 0)=0, & ~t\in (0, T),\\
        \vp(T, x) =  \vp_T(x), & ~x\in (0, L),
        \end{array}
        \right.
    \end{align*}
    coupled by the corresponding transmission conditions \eqref{eq:tc}. Additionally, if $\by\in \X_{\pw, T}^3(0, L)$, for $s\in \{0, 1, 2, 3\}$ there exists $C>0$ such that for any $\vp_T\in \D(\A^*)$ and $g\in L^2(0, T; \D(\A^*))$,
    \begin{align}\label{prop:eq:kdvadjlin-l2Hs}
        \norm{\vp}_{L^2(0, T; H_{\pw}^{s+1}(0, L))}\leq C\big(\norm{\vp_T}_{H_{\pw}^s(0, T)}+\norm{g}_{L^2(0, T; H_{\pw}^{s-1}(0, L))}\big)
    \end{align}
    \end{proposition}
\begin{proof}
    We split $\vp=\vp^1+\vp^2$ where $\vp^1$ solves the system with data $(\vp_T, g)$, and $\vp^2$ solves the system with potential $\by$, initial data $\vp^2(T)=0$ and source $\overline{y}\vp_x^1$. To treat $\vp^1$ we use \cref{prop:wp-kdvadj-Xs} and to treat $\vp^2$, we use a fixed point argument following the same steps of the proof of \cref{prop:NLwp-Xs}. The estimates follow from \cref{prop:wp-kdvadj-Xs} and $\overline{y}\in\mc{X}_{\pw, T}^3(0, L)$ to handle the estimates on the potential. We omit the details.
\end{proof}

\begin{remark}\label{S2:rk:byX3}
    Note that the statement of \cref{prop:wp:adj-linearized} is not empty, as shown by \cref{prop:NLwp-Xs}, and that the constants depend on the norm of $\overline{y}$, be it on $\X_T^0$ or $\X_{\pw, T}^3$.
\end{remark}


\section{A global Carleman estimate}\label{S3}

We first introduce a weight function with internal observation. Let $j\in\ran{0}{N-1}$ be fixed in the sequel and let $\omega_0\Subset I_j$.
\begin{lemma}\label{S3:lem:weightObs}
There exists a continuous function $\beta\in C([0, L])$ such that $\beta|_{\overline{I_k}}\in C^3(\overline{I_k})$ for $k\in\ran{0}{N-1}$, satisfying the following properties:
\begin{enumerate}
    \item\label{lem:wobs1} for some $r>0$, it holds that
    \begin{align*}
    \min_{x\in [0, L]}\beta\geq r\ \text{ and }\ 
        \beta'\neq 0 & \text{ in }\overline{I_j}\setminus \omega_0
    \end{align*}
    and depending on the value of $j$:
    \begin{enumerate}
        \item\label{lem:wobs1a} if $j\not\in \{0, N-1\}$ then
        \begin{align*}
        \left\{\begin{array}{rl}
            \beta'\geq r>0 & \text{ in } \overline{I_k} \text{ for }k\in\ran{0}{j-1},\\
            \beta'\leq -r<0 & \text{ in } \overline{I_k} \text{ for }k\in\ran{j+1}{N-1};
        \end{array}\right.
        \end{align*}
        \item\label{lem:wobs1b} if $j=0$ then
        \begin{align*}
            \beta'\leq -r<0 & \text{ in } \overline{I_k} \text{ for }k\in\ran{1}{N-1}, \text{ and }\ \beta'(0)>0;
        \end{align*}
        \item\label{lem:wobs1c} if $j=N-1$ then
        \begin{align*}
            \beta'\geq r>0 & \text{ in } \overline{I_k} \text{ for }k\in\ran{0}{N-2}, \text{ and }\ \beta'(L)<0;
        \end{align*}
    \end{enumerate}
    \item for some $\kappa\in (1, 2)$ it holds that
    \begin{align}\label{S3:eq:weightkbound}
    \kappa\max_{x\in \overline{I_k}}\beta<2\min_{x\in \overline{I_k}}\beta,\ k\in\ran{0}{N-1};
    \end{align}
    \item at the interface the following transmission conditions hold:
    \begin{align}\label{S3:eq:tcweight}
        \left\{\begin{array}{rrlll}
            \beta(a_k^-)&=&\beta(a_k^+), & ~k\in\ran{1}{N-1},\\
            \sqrt{p_{k-1}}\beta'(a_k^-)&=&\sqrt{p_k}\beta'(a_k^+), & ~k\in\ran{1}{N-1},\\
            p_{k-1}\beta''(a_k^-)&=&p_k\beta''(a_k^+), & ~k\in\ran{1}{N-1}.
            \end{array}\right.
    \end{align}
\end{enumerate}
\end{lemma}

\begin{proof}
We will show the existence of such a weight function $\beta$ by explicitly constructing a piecewise affine function outside the observation interval, with a quartic polynomial on $I_j$, satisfying all the desired properties. Although this is enough to obtain our Carleman estimate, we point out that a more general $\beta$ could be constructed, see for instance \cite[Lemma 1.1, Lemma 2.1]{BDLR07}.
\medskip

\paragraph{\emph{Step 1. Piecewise affine on the left.}} Let us first assume that $j\neq 0$, that is, the observation zone is not located on $I_0$. Take $m_0>0$ and define $\{m_1,\ldots, m_{j-1}\}$ inductively by,
\begin{align*}
    \sqrt{p_k}m_k=\sqrt{p_{k-1}}m_{k-1},\ k\in\ran{1}{j-1}.
\end{align*}
Let us define on $I_k$ for $k\in\ran{0}{j-1}$,
\begin{align*}
    \beta_k(x)=m_k(x-a_k)+c_k,    
\end{align*}
with $m_k>0$ as above and the $c_k$'s are chosen by continuity: take $c_0>0$ and
\begin{align*}
    c_k:=\beta_{k-1}(a_k^-)=m_{k-1}(a_k-a_{k-1})+c_{k-1},\ k\in\ran{1}{j-1}.
\end{align*}
Furthermore, the third transmission condition is automatically satisfied since $\beta_k''=0$ on each $I_k$, and $\beta_k'=m_k>0$ on $I_k$ for $k\in\ran{0}{j-1}$.
\medskip

\paragraph{\emph{Step 2. Quartic polynomial on the observation zone.}} Since $\beta$ is being constructed to be affine outside $I_j$, the third transmission condition forces $\beta_j''(a_j^+)=\beta_j''(a_j^-)=0$. Hence we impose $\beta_j''(x)=m_j(x-a_j)(x-a_{j+1})$, for some $m_j\in \R$ to be determined. We thus have
\begin{align}\label{S3:eq:bj}
    \beta_j'(x)=m_j\mathfrak{I}_j(x)+n_j,
\end{align}
where $\mathfrak{I}_j(x):=\int_{a_j}^x (t-a_j)(t-a_{j+1})dt$ and $n_j:=\tfrac{\sqrt{p_{j-1}}m_{j-1}}{\sqrt{p_j}}>0$ is prescribed by the transmission condition on the left. Let us consider $(\ell_0, \ell_1)\Subset\omega_0$ and choose $m_j$ satisfying
\begin{align*}
    \dfrac{n_j}{|\mathfrak{I}_j(\ell_1)|}<m_j<\dfrac{n_j}{|\mathfrak{I}_j(\ell_0)|},
\end{align*}
so that $\beta_j'(\ell_0)>0$ and $\beta_j'(\ell_1)<0$. As $\mathfrak{I}_j'(x)=(x-a_j)(x-a_{j+1})<0$ for $x\in I_j$, $\beta'_j$ is strictly decreasing on $I_j$, by the intermediate value theorem, the unique zero of $\beta_j'$ must lie inside of $(\ell_0, \ell_1)$, hence $\beta_j'\neq 0$ in $\overline{I_j}\setminus\omega_0$. We then define
\begin{align*}
    \beta_j(x):=m_j\int_{a_j}^x\mathfrak{I}_j(s)ds+n_j(x-a_j)+\beta_{j-1}(a_j^-).
\end{align*}
From the lower bound on $m_j$, that $\mathfrak{I}_j$ is strictly decreasing on $I_j$, and that $|\mathfrak{I}_j(\ell_1)|<|\mathfrak{I}_j(a_{j+1})|$, we have $\beta_j'(a_{j+1}^-)<0$. If $j=N-1$, the construction ends here and it ensures that $\beta'(L)<0$. Otherwise, we continue by constructing a piecewise affine function on the right-side of the observation zone.
\medskip

\paragraph{\emph{Step 3. Piecewise affine on the right.}} For $j\neq N-1$, we set $m_{j+1}:=\tfrac{\sqrt{p_j}}{\sqrt{p_{j+1}}}\beta_j'(a_{j+1}^-)<0$ and we define inductively $m_{k+1}=\tfrac{\sqrt{p_{k}}}{\sqrt{p_{k+1}}}m_k<0$ for $k\in\ran{j+1}{N-2}$. For $k\in\ran{j+1}{N-1}$ we define
\begin{align*}
    \beta_k(x)=m_k(x-a_k)+c_k,
\end{align*}
with $c_{j+1}=\beta_j(a_{j+1}^-)$ and the remaining $c_k$'s chosen by continuity as in \emph{Step 1}. Note that $\beta_k'=m_k<0$ on $I_k$ for each $k\in\ran{j+1}{N-1}$.

We then set $\beta$ to be defined on each $I_k$ by $\beta_{|_{I_k}}:=\beta_k$ for $k\in\ran{0}{N-1}$ with the obvious traces defined by continuity. By construction, the transmission conditions \eqref{S3:eq:tcweight} are clearly satisfied by $\beta$.
\medskip

\paragraph{\emph{Step 4. Final bounds.}} Let us set $\widetilde{\beta}=\beta+K$ with $K>0$ to be chosen below. Observe that the transmission conditions and the sign of the derivatives remain unchanged under this shift.

Let $C:=\max_{x\in [0, L]}\beta$ and $c:=\min_{x\in[0, L]}\beta$. Note that for each $k\in\ran{0}{N-1}$ it holds
\begin{align*}
    \min_{\overline{I_k}}\widetilde{\beta}\geq K+c\ \text{ and }\ \max_{\overline{I_k}}\widetilde{\beta}\leq K+C.
\end{align*}
A sufficient condition for \eqref{S3:eq:weightkbound} to hold is to choose $K$ large enough so that $K+c>\frac{\kappa}{2-\kappa}(C-c).$
By setting
\begin{align*}
    r:=\left\{\begin{array}{cc}
    \min\Big\{K+c,\displaystyle \min_{k\leq j-1} m_k, \min_{j+1\leq k\leq N-1} (-m_k)\Big\}, & j\neq N-1,\\
    \min\Big\{K+c, \displaystyle\min_{k\leq N-2} m_k\Big\}, & j=N-1,
    \end{array}\right.
\end{align*}
we obtain the desired bounds. 
\medskip

\paragraph{\emph{Step 5. Right endpoint case.}} If $j=0$, we perform a similar construction starting off from $I_0$ as in \emph{Step 2} ensuring $\beta'(0)>0$ and $\beta'\neq 0$ on $\overline{I_0}\setminus\omega_0$, and then as in \emph{Step 3} constructing piecewise linear functions with negative slope. The proof finishes by relabeling $\widetilde{\beta}$ by $\beta$.
\end{proof}

\begin{remark}
    Since we are deriving a two-parameter Carleman estimate, a condition on the second derivative of the weight function is not required, unlike in the one-parameter case. However, we retain the second-order transmission condition, as it simplifies some computations later on.
\end{remark}

Let us set $Q'=(0, T)\times ((0, L)\setminus\Gamma)$. The weight functions \eqref{S1:intro:weights} satisfy the following identities
\begin{align*}
\begin{array}{lll}
    \partial_x\eta=-\ld\beta'\xi,\hspace{0.8cm} & \partial_x\xi=\ld\beta'\xi,\hspace{0.8cm} & \text{ in } Q',\\
    \partial_t\eta=\dfrac{2t-T}{t(T-t)}\eta, &  \partial_t\xi=\dfrac{2t-T}{t(T-t)}\xi, & \text{ in } Q.
    \end{array}
\end{align*}
We have that for each $k\in \N$, there exists $C>0$ independent of $\ld>0$ such that
\begin{align}\label{S3:ineq:weightx}
    |\partial_x^k\eta(t, x)|+|\partial_x^k\xi(t, x)|\leq C(\ld^k+1)\xi(t, x),\ (t, x)\in Q'.
\end{align}
Furthermore, due to the properties of $\beta$, there exists a constant $C=C(T)>0$ such that
\begin{align}\label{S3:ineq:weightt}
    C^{-1}\leq \xi(t, x)\ \text{   and   }\ |\partial_t\eta(t, x)|+|\partial_t\xi(t, x)|\leq C\xi^2(t, x),\ (t, x)\in Q.
\end{align}
These estimates will give us the heuristics to identify the dominating and lower order terms coming from the integration by parts later on.

\subsection{Proof of \cref{thm:carleman_omega}} Let $s>0$ and define $\mc{V}_s=\{e^{-s\eta}u : u\in\mc{V}\}$. For $u\in\mc{V}$ set $w=e^{-s\eta}u$ and introduce the conjugate operator
\begin{align*}
    \mc{L}_\eta w=e^{-s\eta}\mc{L}(e^{s\eta}w)=(\mc{L}_1+\mc{L}_2+\mc{R})w
\end{align*}
where
\begin{align*}
    \mc{L}_1w&=w_t+3ps^2\eta_x^2w_x+pw_{xxx}+3pms^2\eta_x\eta_{xx}w,\\
    \mc{L}_2w&=ps^3\eta_x^3w+3ps\eta_xw_{xx}+3sw_x(p\eta_x)_x,
\end{align*}
and
\begin{align*}
    \mc{R}w=bs\eta_x w+bw_x+ps\eta_{xxx}w+3ps^2\eta_x\eta_{xx}w+dw+s\eta_tw-3sp_x\eta_xw_x-3pms^2\eta_x\eta_{xx}w,
\end{align*}
for some constant $m>0$, to be chosen later. By \cref{S3:lem:weightObs}, since $\beta$ satisfies the transmission conditions, the conjugate function satisfies them as well:
\begin{align}\label{eq:TCconjugate}
    \left\{\begin{array}{rrlll}
    w(t, a_k^-)&=&w(t, a_k^+), & ~t\in (0, T),~k\in\ran{1}{N-1},\\
    \sqrt{p_{k-1}}w_x(t, a_k^-)&=&\sqrt{p_k}w_x(t, a_k^+), & ~t\in (0, T),~k\in\ran{1}{N-1},\\
    p_{k-1}w_{xx}(t, a_k^-)&=&p_kw_{xx}(t, a_k^+), & ~t\in (0, T),~k\in\ran{1}{N-1}.
    \end{array}\right.
\end{align}
Taking the $L^2-$norm to $\mc{L}_1w+\mc{L}_2w=\mc{L}_\eta w-\mc{R}w$ we obtain
\begin{align*}
    \norm{\mc{L}_1w}_{L^2(Q)}^2+\norm{\mc{L}_2w}_{L^2(Q)}^2+2\inn{\mc{L}_1 w,\mc{L}_2 w}_{L^2(Q)}\leq 2\norm{\mc{L}_\eta w}_{L^2(Q)}^2+2\norm{\mc{R}w}_{L^2(Q)}^2.
\end{align*}
The introduction of a well chosen parameter $m>0$ for the last term in $\mc{L}_1$ originates from an idea of Fursikov and Imanuvilov \cite{FI96} in the parabolic case, and allows us obtain positivity for the crossed product term without relying on the square terms $\norm{\mc{L}_1w}_{L^2}^2$ and $\norm{\mc{L}_2w}_{L^2}^2$. We refer to Le Rousseau, Lebeau and Robbiano's book on Carleman estimates \cite[Section 3.7]{LeRLR22} for a thorough analysis of this term.
\subsubsection{Double product term} To fix notation, in what follows the symbol $|_0^L$ denotes the evaluation at the end points considering the interface, namely,  $\mu|_0^L:=\sum_{k=0}^{N-1}\mu|_{a_k}^{a_{k+1}}$. Henceforth, we will write
\begin{align*}
    \mu\big|_0^L=\mu(L)-\mu(0)-\sum_{a\in \Gamma} [\mu]_a. 
\end{align*}
Denote by $I_{ij}$ for $i\in\ran{1}{4}$, $j\in\ran{1}{3}$ the $ij-$term of the $L^2-$product $\inn{\mc{L}_1 w,\mc{L}_2 w}_{L^2(Q)}$. We have that $w(0, t)=w(L, t)=0$ for all $t\in (0, T)$ and $w(x, 0)=w(x, T)=0$ for all $x\in (0, L)$. In what follows, for each term $I_{ij}$ we will perform several integration by parts and we will write once explicitly all the terms. Then we will gather them into two groups: the distributed and boundary-interface terms. Within each group we will split the terms with respect to the powers of $s$, $\ld$ and $\xi$ into dominating and lower-order terms: the former will produce the weighted norms we are looking for and the latter will be absorbed for large $s$ and $\ld$. We perform the integration by parts below:
\begin{align*}
    I_{11}=\iint_Q ps^3\eta_x^3 ww_t dxdt=\dfrac{3}{2}s^3\ld^3\iint_Q p\beta_x^3\xi^2\xi_t|w|^2dxdt,
\end{align*}
\begin{align*}
    I_{12}&=3s\iint_Q p\eta_x w_{xx}w_tdxdt\\
    &=-I_{13}-\dfrac{3}{2}s\ld\iint_Q p\beta_x\xi_{t} |w_{x}|^2dxdt-3s\ld\left.\int_0^T p\beta_x\xi w_tw_xdt\right|_0^L,
\end{align*}
\begin{align*}
    I_{21}&=3s^5\iint_Q p^2\eta_x^5ww_{x}dxdt\\
    &=\dfrac{3}{2}s^5\ld^5\iint_Q (p^2\beta_x^5)_x\xi^5|w|^2dxdt+\dfrac{15}{2}s^5\ld^6\iint_Q p^2 \beta_x^6\xi^5|w|^2dxdt-\left.\dfrac{3}{2}s^5\ld^5\int_0^T p^2\beta_x^5\xi^5|w|^2dt\right|_0^L,
\end{align*}
\begin{align*}
    I_{22}&=9s^3\iint_Q p^2\eta_x^3 w_xw_{xx}dxdt\\
    &=\dfrac{9}{2}s^3\ld^3\iint_Q (p^2\beta_x^3)_x\xi^3 |w_x|^2dxdt+\dfrac{27}{2}s^3\ld^4\iint_Q p^2 \beta_x^4\xi^3|w_x|^2dxdt-\dfrac{9}{2}s^3\ld^3\left.\int_0^T p^2\beta_x^3\xi^3|w_x|^2dt\right|_0^L,\\
\end{align*}
\begin{align*}
    I_{23}&=9\iint_Q ps^3\eta_x^2(p\eta_{x})_x|w_x|^2dxdt\\
    &=-9s^3\ld^3\iint_Q p \beta_x^2(p\beta_{x})_x\xi^3|w_x|^2dxdt-9s^3\ld^4\iint_Q p^2 \beta_x^4\xi^3|w_x|^2dxdt,
\end{align*}
\begin{align*}
    I_{31}&=s^3\iint_Q p^2\eta_x^3w w_{xxx}dxdt\\
    &=\begin{multlined}[t][15.4cm]\dfrac{s^3\ld^3}{2}\iint_Q (p^2\beta_x^3\xi^3)_{xxx} |w|^2dxdt-\dfrac{3}{2}s^3\ld^3\iint_Q (p^2\beta_x^3)_x\xi^3|w_x|^2dxdt-\dfrac{9}{2}s^3\ld^4\iint_Q p^2\beta_x^4\xi^3|w_x|^2dxdt\\-\left.\dfrac{s^3\ld^3}{2}\int_0^T (p^2\beta_x^3\xi^3)_{xx}|w|^2dt\right|_0^L+\left.s^3\ld^3\int_0^T (p^2\beta_x^3\xi^3)_xww_xdt\right|_0^L+\left.\dfrac{s^3\ld^3}{2}\int_0^T p^2\beta_x^3\xi^3|w_x|^2dt\right|_0^L\\-\left.s^3\ld^3\int_0^T p^2\beta_x^3\xi^3ww_{xx}dt\right|_0^L,\end{multlined}
\end{align*}
\begin{align*}
    I_{32}&=3s\iint_Q p^2\eta_x w_{xx}w_{xxx}dxdt\\
    &=\dfrac{3}{2}s\ld\iint_Q (p^2\beta_{x})_x\xi |w_{xx}|^2dxdt+\dfrac{3}{2}s\ld^2\iint_Q p^2 \beta_{x}^2\xi |w_{xx}|^2dxdt-\left.\dfrac{3}{2}s\ld\int_0^T p^2 \beta_x\xi |w_{xx}|^2dt\right|_0^L,
\end{align*}
\begin{align*}
    I_{33}&=3s\iint_Q p(p\eta_{x})_xw_xw_{xxx}dxdt\\
    &=\begin{multlined}[t][15cm]-\dfrac{3}{2}s\ld\iint_Q (p(p\beta_{x}\xi)_x)_{xx}|w_x|^2+3s\ld\iint_Q p(p\beta_{x})_x\xi|w_{xx}|^2dxdt+3s\ld^2\iint_Q p^2\beta_{x}^2\xi|w_{xx}|^2dxdt\\+\left.\dfrac{3}{2}s\ld\int_0^T (p(p\beta_{x}\xi)_x)_{x}|w_x|^2dt\right|_0^L-\left.3s\ld\int_0^T p(p\beta_{x}\xi)_x w_xw_{xx}dt\right|_0^L,\end{multlined}
\end{align*}
\begin{align*}
    I_{41}&=3ms^5\iint_Q p^2\eta_x^4\eta_{xx}|w|^2dxdt\\
    &=-3ms^5\ld^5\iint_Q p^2\beta_x^4\beta_{xx}\xi^5|w|^2dxdt-3ms^5\ld^6\iint_Q p^2\beta_x^6\xi^6|w|^2dxdt,
\end{align*}
\begin{align*}
    I_{42}&=9ms^3\iint_Q p^2\eta_x^2\eta_{xx}ww_{xx}dxdt\\
    &=\begin{multlined}[t][15cm]-\dfrac{9}{2}ms^3\ld^3\iint_Q (p^2\beta_x^2\beta_{xx}\xi^3)_{xx}|w|^2dxdt-\dfrac{9}{2}ms^3\ld^4\iint_Q (p^2\beta_x^4\xi^3)_{xx}|w|^2dxdt\\+9ms^3\ld^3\iint_Q p^2 \beta_x^2\beta_{xx}\xi^3|w_x|^2dxdt+9ms^3\ld^4\iint_Q p^2 \beta_x^4\xi^3|w_x|^2dxdt\\+\left.\dfrac{9}{2}ms^3\ld^3\int_0^T(p^2\beta_x^2\beta_{xx}\xi^3)_x|w|^2dt\right|_0^L+\left.\dfrac{9}{2}ms^3\ld^4\int_0^T(p^2\beta_x^2\xi^3)_x|w|^2dt\right|_0^L
    \\-\left.9ms^3\ld^3\int_0^T p^2\beta_x^2\beta_{xx}\xi^3ww_xdt\right|_0^L-\left.9ms^3\ld^4\int_0^T p^2\beta_x^4\xi^3ww_xdt\right|_0^L,\end{multlined}
\end{align*}
\begin{align*}
    I_{43}&=9ms^3\iint_Q p(p\eta_x)_x\eta_x\eta_{xx} ww_{x}dxdt\\
    &=\begin{multlined}[t][15cm]\dfrac{9}{2}ms^3\ld^3\iint_Q\big(p(p\beta_x\xi)_x\beta_x\beta_{xx}\xi^2\big)_x|w|^2dxdt+\dfrac{9}{2}ms^3\ld^4\iint_Q\big(p(p\beta_x\xi)_x\beta_x^3\xi^2\big)_x|w|^2dxdt+\\
    -\left.\dfrac{9}{2}ms^3\ld^3\int_0^Tp(p\beta_x\xi)_x\beta_x\beta_{xx}\xi^2|w|^2dt\right|_0^L-\left.\dfrac{9}{2}ms^3\ld^4\int_0^Tp(p\beta_x\xi)_x\beta_x^3\xi^2|w|^2dt\right|_0^L.\end{multlined}
\end{align*}

\subsubsection{Gathering terms} We split the double product terms as follows
\begin{multline*}
    \inn{\mc{L}_1w, \mc{L}_2w}_{L^2(Q)}=\left(\dfrac{15}{2}-3m\right)s^5\ld^6\iint_Q p^2\beta_x^6\xi^5|w|^2dxdt\\+9ms^3\ld^4\iint_Q p^2\beta_x^4\xi^3|w_x|^2dxdt+\dfrac{9}{2}s\ld^2\iint_Q p^2\beta_x^2\xi|w_{xx}|^2dxdt+\kD^{low}+\kB,
\end{multline*}
where $\kD^{low}$ and $\kB$ gather the lower order distributed and boundary-interface terms, respectively. For the boundary-interface terms, without taking into account any of the boundary conditions nor the transmission conditions, we have
\begin{multline*}
    \kB=-\left.\dfrac{3}{2}s^5\ld^5\int_0^T p^2\beta_x^5\xi^5|w|^2dt\right|_0^L+\left.\dfrac{9}{2}ms^3\ld^4\int_0^T(p^2\beta_x^2\xi^3)_x|w|^2dt\right|_0^L\\-\left.\dfrac{9}{2}ms^3\ld^4\int_0^Tp(p\beta_x\xi)_x\beta_x^3\xi^3|w|^2dt\right|_0^L-\left.\dfrac{s^3\ld^3}{2}\int_0^T (p^2\beta_x^2\xi^3)_{xx}|w|^2dt\right|_0^L\\+\left.\dfrac{9}{2}ms^3\ld^3\int_0^T(p^2\beta_x^2\beta_{xx}\xi^3)_x|w|^2dt\right|_0^L-\left.\dfrac{9}{2}ms^3\ld^3\int_0^Tp(p\beta_x\xi)_x\beta_x\beta_{xx}\xi^2|w|^2dt\right|_0^L\\
    -4s^3\ld^3\left.\int_0^T p^2\beta_x^3\xi^3|w_x|^2dt\right|_0^L+\left.\dfrac{3}{2}s\ld\int_0^T [p(p\beta_{x}\xi)_x]_{x}|w_x|^2dt\right|_0^L
    -\left.\dfrac{3}{2}s\ld\int_0^T p^2 \beta_x\xi |w_{xx}|^2dt\right|_0^L\\-\left.3s\ld\int_0^T p(p\beta_{x}\xi)_x w_xw_{xx}dt\right|_0^L-\left.s^3\ld^3\int_0^T p^2\beta_x^3\xi^3ww_{xx}dt\right|_0^L
    +\left.s^3\ld^3\int_0^T (p^2\beta_x^3\xi^3)_xww_xdt\right|_0^L\\-\left.9ms^3\ld^3\int_0^T p^2\beta_x^2\beta_{xx}\xi^3ww_xdt\right|_0^L-\left.9ms^3\ld^4\int_0^T p^2\beta_x^4\xi^3ww_xdt\right|_0^L-3s\ld\left.\int_0^T p\beta_x\xi w_tw_xdt\right|_0^L.
\end{multline*}
In view of the boundary and the transmission conditions, we split $\kB$ as follows
\begin{align*}
    \kB=\kB_L+\kB_0+\kB^*+\kB_\Gamma,
\end{align*}
where each one of these terms is described below. To slightly simplify the notation below and keep better track of the dominant powers, let us define $\wxi:=s\ld\xi$. First of all, $\kB_L$ and  $\kB_0$ correspond to those terms at $x=L$ and $0$, respectively, that have fixed sign
\begin{gather*}
    \kB_L=-4\int_0^T \big(p^2\beta_x^3\wxi^3|w_x|^2\big)\big|_{x=L}dt-\dfrac{3}{2}\int_0^T \big(p^2 \beta_x\wxi|w_{xx}|^2\big)\big|_{x=L}dt,\\
    \kB_0=4\int_0^T \big(p^2\beta_x^3\wxi^3|w_x|^2\big)\big|_{x=0}dt+\dfrac{3}{2}\int_0^T \big(p^2\beta_x\wxi |w_{xx}|^2\big)\big|_{x=0}dt,
\end{gather*}
and $\kB^*$ corresponds to the terms without fixed sign at the boundary
\begin{multline*}
    \kB^*=\dfrac{3}{2}s\ld\int_0^T \big((p(p\beta_{x}\xi)_x)_x|w_x|^2\big)\big|_{x=L}dt-\dfrac{3}{2}s\ld\int_0^T \big((p(p\beta_{x}\xi)_x)_x|w_x|^2\big)\big|_{x=0}dt\\+3s\ld\int_0^T  \big(p(p\beta_{x}\xi)_xw_xw_{xx})\big|_{x=0}dt-3s\ld\int_0^T  (p(p\beta_{x}\xi)_xw_xw_{xx})\big|_{x=L}dt.
\end{multline*}
In the same spirit as before, we split the terms at the interface as follows
\begin{align*}
    \kB_\Gamma=\kB_\Gamma^{dom}+\kB_\Gamma^{low}=\sum_{a\in \Gamma}\big(\kB_\Gamma^{dom}(a)+\kB_\Gamma^{low}(a)\big),
\end{align*}
where
\begin{multline*}
    \kB_\Gamma^{dom}(a)=\dfrac{3}{2}\int_0^T [p^2\beta_x^5\wxi^5|w|^2]_{a}dt+4\int_0^T [p^2\beta_x^3\wxi^3|w_x|^2]_adt\\+\dfrac{3}{2}\int_0^T[p^2\beta_x\wxi|w_{xx}|^2]_{a}dt+\int_0^T [p^2\beta_x^3\wxi^3 ww_{xx}]_adt+3\int_0^T [p\beta_x\wxi w_tw_x]_adt,
\end{multline*}
and $\kB_\Gamma^{low}$ gathers the remaining terms at the interface, which will be shown to be of lower order with respect to $\kB_\Gamma^{dom}$.

\subsubsection{Estimates for the distributed terms} Henceforth the generic constant depends on $L$, $T$, $\rho_0$, $\rho_1$, $\norm{\beta}_{C^3}$, $r$, $s_0$ and $\ld_0$, where $s_0$ and $\ld_0$ will be chosen later. Let us introduce
\begin{gather*}
    \norm{w}_{s,\ld, \xi}^2:=\iint_{Q}\Big(s^5\ld^6\xi^5|w|^2+s^3\ld^4\xi^3|w_x|^2+s\ld^2\xi|w_{xx}|^2\Big)dxdt,\\
    \obs^2:=\iint_{(0, T)\times\omega_0}\Big(s^5\ld^6\xi^5|w|^2+s^3\ld^4\xi^3|w_x|^2+s\ld^2\xi|w_{xx}|^2\Big)dxdt,
\end{gather*}
which stand as the weighted norm in the interior and the weighted observation from $(0, T)\times\omega_0$, respectively.
Let us fix $m\in(0, 5/2)$. Using \cref{S3:lem:weightObs}-\cref{lem:wobs1}, the dominating terms are bounded from below as follows
\begin{multline*}
    \left(\dfrac{15}{2}-3m\right)s^5\ld^6\iint_Q p^2\beta_x^6\xi^5|w|^2dxdt+9ms^3\ld^4\iint_Q p^2\beta_x^4\xi^3|w_x|^2dxdt\\+\dfrac{9}{2}s\ld^2\iint_Q p^2\beta_x^2\xi|w_{xx}|^2dxdt
    \gtrsim \norm{w}_{s, \ld, \xi}^2-\obs^2,
\end{multline*}
whilst for $\kD^{low}$, we additionally use \eqref{S3:ineq:weightx} and Young's inequality (to distribute powers of $s$ and $\ld$ according to the number of derivatives of $w$), to obtain
\begin{align*}
    |\kD^{low}|\lesssim\left(\dfrac{1}{s}+\dfrac{1}{s^2}+\dfrac{1}{\ld}\right)\norm{w}_{s,\ld, \xi}^2.
\end{align*}
Therefore, by choosing $s_0$ and $\ld_0$ large enough, we get
\begin{align}\label{S3:ineq:L2prod1}
    \norm{w}_{s,\ld, \xi}^2+\kB\lesssim \inn{\mc{L}_1 w,\mc{L}_2 w}_{L^2(Q)}+\obs^2,
\end{align}
for all $s\geq s_0$ and $\ld\geq \ld_0$. Additionally, for the definition of $\mc{R}$, we observes that the highest powers are $s^4\ld^6$ for the zero order term and $s^2\ld^2$ for the first order term. From the regularity assumptions of $p$, $b$ and $d$ we obtain we observe that for the residue term one has the estimate
\begin{align*}
    \norm{\mc{R} w}_{L^2(Q)}^2\lesssim s^4\ld^6\iint_Q \xi^4|w|^2dxdt+s^2\ld^2\iint_Q \xi^2|w_x|^2dxdt\lesssim \left(\dfrac{1}{s}+\dfrac{1}{s\ld^2}\right)\norm{w}_{s, \ld, \xi}^2.
\end{align*}

\subsubsection{Estimates for the boundary terms} From the properties of $\beta$ given by \cref{S3:lem:weightObs}-\cref{lem:wobs1}, no matter where $\omega_0$ is located, we have that $\beta_x(0)$ and $-\beta_x(L)$ are both positively bounded from below, from which follows
\begin{align*}
    \sum_{x_b\in\{0, L\}}\int_0^T \big(\wxi^3|w_x|^2+\wxi|w_{xx}|^2\big)\big|_{x=x_b}dt\lesssim \kB_0+\kB_L.
\end{align*}
To treat the terms contained in $\kB^*$, as before, using inequality \eqref{S3:ineq:weightx}, we get
\begin{align*}
        \left|\dfrac{3}{2}s\ld\int_0^T \big((p(p\beta_{x}\xi)_x)_{x}|w_x|^2\big)\big|_{x=L}dt\right|\lesssim \dfrac{1}{s^2}\int_0^T \big(\wxi^3|w_x|^2\big)\big|_{x=L}dt,
\end{align*}
and using Young's inequality (to distribute powers of $s$ and $\ld$ accordingly),
\begin{align*}
    \left|3s\ld\int_0^T \big(p^2(\beta_{x}\xi)_xw_xw_{xx}\big)\big|_{x=L}dt\right|\lesssim \dfrac{1}{s}
    \int_0^T \big(|\wxi^3|w_x|^2+\wxi|w_{xx}|^2\big)\big|_{x=L}dt.
\end{align*}
The analogous bounds hold for the terms evaluated at $x=0$. We readily get for $s\geq s_0$,
\begin{align*}
    |\kB^*|\lesssim \dfrac{1}{s}\sum_{x_b\in\{0, L\}}\int_0^T \big(\wxi^3|w_x|^2+\wxi|w_{xx}|^2\big)\big|_{x=x_b}dt.
\end{align*}
Then taking $s_0$ large enough, from \eqref{S3:ineq:L2prod1}, for any $s\geq s_0$ and $\ld\geq \ld_0$ it holds
\begin{align*}
    \norm{w}_{s,\ld, \xi}^2+\sum_{x_b\in\{0, L\}}\int_0^T \big(\wxi^3|w_x|^2+\wxi|w_{xx}|^2\big)\big|_{x=x_b}dt+\kB_\Gamma\lesssim \inn{\mc{L}_1 w,\mc{L}_2 w}_{L^2(Q)}+\obs^2.
\end{align*}

\subsubsection{Treatment of the terms at the interface} Recalling that $\wxi:=s\ld\xi$, we introduce the weighted norm at the interface
\begin{align*}
    |w|_{\Gamma, s, \ld, \xi}^2=\sum_{a\in \Gamma}\int_0^T\left(\wxi^5|w|^2+\wxi^3|\sqrt{p}w_x|^2+\wxi|pw_{xx}|^2\right)(t, a)dt.
\end{align*}
First, since $\beta$ and $w$ satisfy the transmission conditions \eqref{S3:eq:tcweight} and \eqref{eq:TCconjugate}, respectively, we have
\begin{align*}
    \int_0^T [p\beta_x\xi w_tw_x]_adt=0,\ \forall a\in \Gamma.
\end{align*}
Now, the transmission conditions allow us to rewrite the remaining terms of $\kB_\Gamma^{dom}$ as follows
\begin{gather*}
    \dfrac{3}{2}\int_0^T[p^2\beta_x^5\wxi^5|w|^2]_{a}dt=\dfrac{3}{2}\int_0^T \big(p^2\beta_x^4\big)|_{a^+}[\beta_x]_a\wxi^5(t, a)|w(t, a)|^2dt,\\
    4\int_0^T [p^2\beta_x^3\wxi^3|w_x|^2]_{a}dt=4\int_0^T \big(p\beta_x^2\big)|_{a^+}[\beta_x]_a\wxi^3(t, a)|pw_x(t, a)|^2dt,\\
    \dfrac{3}{2}\int_0^T[p^2\beta_x\wxi|w_{xx}|^2]_{a}dt=\dfrac{3}{2}\int_0^T[\beta_x]_a\wxi(t, a)|pw_{xx}(t, a)|^2dt,\\
    \int_0^T [p^2\beta_x^3\wxi^3 ww_{xx}]_{a}dt=\int_0^T \big(p\beta_x^2\big)|_{a^+}[\beta_x]_a\wxi^3(t, a)pw_{xx}(t, a)w(t, a)dt.
\end{gather*}
Let us define the vector function $\vec\bw_{s, \ld}: (0, T)\times\Gamma\to \R^3$ by
\begin{align*}
    \vec\bw_{s, \ld}(t, a)=\big(\wxi^2(t, a)w(t, a), \wxi(t, a)\sqrt{p}w_x(t, a), p w_{xx}(t, a)\big)^\tr.
\end{align*}
By the above computations we can write
\begin{align*}
    \kB_\Gamma^{dom}(a)=\dfrac{3}{2}\int_0^T \wxi(t, a)\big(\boldsymbol{A}\vec\bw_{s, \ld}(t, a), \vec\bw_{s, \ld}(t, a)\big)_{\R^3}dt,
\end{align*}
where $\bA$ is defined by
\begin{align*}
    \bA(a):=\begin{pmatrix} [p^2\beta_x^5]_a & 0 & \tfrac{1}{3}[p\beta_x^3]_a\\
    0 & \tfrac{8}{3}[p\beta_x^3]_a & 0\\
    \tfrac{1}{3}[p\beta_x^3]_a & 0 & [\beta_x]\end{pmatrix}.
\end{align*}
No matter where $\omega_0$ is located, since $\beta$ satisfies the transmission conditions and Hypothesis \ref{assumM} is enforced, we have $[\beta_x]_a>0$ for any $a\in \Gamma$. Then, looking at the sign of each minor of $\bA(a)$, we have
\begin{align*}
    & \Delta_1(a)=[p^2\beta_x^5]_a=(p\beta_x^2)_{|_{a^-}}^2[\beta_x]_a>0,\\
    &\Delta_2(a)=\det\begin{pmatrix}
        [p^2\beta_x^5]_a & 0\\
    0 & \tfrac{8}{3}[p\beta_x^3]_a
    \end{pmatrix}=\frac{8}{3}(p\beta_x^2)_{|_{a^-}}^3[\beta_x]_a^2>0,\\
    &\Delta_3(a)=\det\bA(a)=\frac{8}{3}[p\beta_x^3]_a\det\begin{pmatrix}
        [p^2\beta_x^5]_a & \tfrac{1}{3}[p\beta_x^3]_a\\
    \tfrac{1}{3}[p\beta_x^3]_a & [\beta_x]_a
    \end{pmatrix}=\frac{64}{27}(p\beta_x^2)_{|_{a^-}}^3[\beta_x]_a^3>0.
\end{align*}
Given that $\bA(a)$ is a symmetric matrix with strictly positive minors, by using Sylvester's criterion, it follows that $\bA(a)$ is a positive definite matrix. We thus choose $\gamma>0$ to be the minimum over the lower bounds of the associated quadratic form to $\bA(a)$ running over $a\in\Gamma$. We thus obtain, uniformly in $a\in\Gamma$, that
\begin{align*}
    \int_0^T \wxi(t, a)\big(\boldsymbol{A}\vec\bw_{s, \ld}(t, a), \vec\bw_{s, \ld}(t, a)\big)_{\R^3}dt\geq \gamma \int_0^T \wxi(t, a)|\vec\bw_{s, \ld}(t, a)|_{\R^3}^2dt.
\end{align*}
As we did  with the boundary terms, by using inequality \eqref{S3:ineq:weightx} and Young's inequality, we get
\begin{align*}
    |\kB_\Gamma^{low}(a)|\lesssim \left(\dfrac{1}{s\ld^2}+\dfrac{1}{s^2\ld}+\dfrac{1}{s^2}+\dfrac{1}{\ld}\right)\int_0^T\Big(\wxi^5|w|^2+\wxi^3|\sqrt{p}w_x|^2+\wxi|pw_{xx}|^2\Big)(t, a)dt.
\end{align*}
Since the above estimates are uniform with respect to $a\in \Gamma$, by choosing $s_0$ and $\ld_0$ large enough yields that, for any $s\geq s_0$ and $\ld\geq \ld_0$, 
\begin{align*}
    |w|_{\Gamma, s, \ld, \xi}^2\lesssim \kB_\Gamma.
\end{align*}

\subsubsection{Back to the original variable} Gathering the inequalities obtained in the previous steps, we proved the following estimate for the conjugated operator.

\begin{proposition}\label{S3:prop:carleman-conj} Let $(\omega, p)$ satisfy Hypothesis \ref{assumM} and let $\omega_0\Subset\omega$. There exist $s_0>0$, $\ld_0>0$ and a constant $C>0$ depending on $L$, $T$, $p$, $s_0$, $\ld_0$ and $\norm{\beta}_{C^3([0, L]\setminus\Gamma)}$, such that for all $w\in \mc{V}_s$ we have
\begin{align*}
C\Big(\norm{\mc{L}_1 w}_{L^2(Q)}^2+\norm{\mc{L}_2 w}_{L^2(Q)}^2+\norm{w}_{s, \ld, \xi}^2+|w|_{\Gamma, s, \ld, \xi}^2\Big)\leq \norm{\mc{L}_\eta w}_{L^2(Q)}^2+\obs^2
\end{align*}
for any $s\geq s_0$ and $\ld\geq \ld_0.$
\end{proposition}

Now we go back to the original variable. Recall that $u=e^{s\eta}w$ belongs to $\mc{V}$ and $\mc{L}_\eta w=e^{-s\eta}\mc{L}u$. Straightforward computations lead us to
\begin{align*}
    e^{-2s\eta}|u_x|^2&\lesssim |w_x|^2+s^2\ld^2\xi^2|w|^2,\\
    e^{-2s\eta}|u_{xx}|^2&\lesssim |w_{xx}|^2+s^2\ld^2\xi^2|w_x|^2+s^4\ld^4\xi^2|w|^2,
\end{align*}
for all $(t, x)\in Q'$, from which we get
\begin{align*}
    \iint_Q e^{-2s\eta}\Big(s^5\ld^6\xi^5|u|^2+s^3\ld^4\xi^3|u_x|^2+s\ld^2\xi|u_{xx}|^2\Big)dxdt
    \lesssim \norm{w}_{s, \ld, \xi}^2.
\end{align*}
Using $w=e^{-s\eta}u$, similar estimates lead us to
\begin{align*}
    \obs^2\lesssim \iint_{(0, T)\times\omega_0} e^{-2s\eta}\Big(s^5\ld^6\xi^5|u|^2+s^3\ld^4\xi^3|u_x|^2+s\ld^2\xi|u_{xx}|^2\Big)dxdt.
\end{align*}
From the above estimates and using $\omega_0\Subset\omega$, \cref{S3:prop:carleman-conj} directly implies \cref{thm:carleman_omega}.


\section{Control to the trajectories}\label{S4}

The aim of this section is to prove the controllability result \cref{thm:NLKdV-control}. To this end we will consider the following two relevant systems. The first one corresponds to the linearized system to \eqref{eq:NLKdV-control} around the aimed trajectory $\by$, which is
\begin{align}\label{eq:KdV-lin-control}
\left\{
\begin{array}{rll}
z_t+p(x)z_{xxx}+z_x+(\by z)_x=h+\mathbbm{1}_\omega v, & ~(t, x)\in (0, T)\times (0, L),\\
z(t, 0)=z(t, L)=z_{x}(t, L)=0, & ~t\in (0, T),\\
z(0, x)=z_0(x), & ~x\in (0, L),
\end{array}
\right.
\end{align}
coupled by \eqref{eq:tc}, where $v\in L^2(0, T; L^2(0, L))$ is the control and $h$ is a source in some appropriate weighted space. The second relevant system corresponds to the adjoint system associated to \eqref{eq:KdV-lin-control}
\begin{align}\label{eq:KdV-lin-adj}
\left\{
\begin{array}{rll}
-\vp_t-p(x)\vp_{xxx}-\vp_x-\by \vp_x = g, & ~(t, x)\in (0, T)\times (0, L),\\
\vp(t, 0)=\vp(t, L)=\vp_{x}(t, 0)=0, & ~t\in (0, T),\\
\vp(T, x) =  \vp_T(x), & ~x\in (0, L).
\end{array}
\right.
\end{align}
coupled by \eqref{eq:tc}, with appropriate initial data $\vp_T$ and source term $g$. The strategy follows a classical duality argument which is briefly described below:
\begin{enumerate}
    \item We establish a suitable Carleman estimate for the adjoint system \eqref{eq:KdV-lin-adj}.
    \item By means of the Carleman estimate, we obtain an observability inequality for \eqref{eq:KdV-lin-adj}. We then employ a variational approach to establish the null controllability of the linearized system \eqref{eq:KdV-lin-control} with a right-hand side decaying near $t=T$.
    \item We then apply a local inversion result in a suitable functional setting, inherited from the variational approach, to obtain the null controllability of the nonlinear system \eqref{eq:NLKdV-control}.
\end{enumerate}
In this section, we closely follow Cerpa, Montoya and Zhang \cite{CMZ20} and show their arguments can be adapted to the piecewise constant case. The main point being the regularity estimates provided by \cref{prop:wp:adj-linearized}, which combined with the Carleman estimate of \cref{thm:carleman_omega} will allow us to obtain a suitable one-parameter Carleman estimate and henceforth carry out the strategy.

\subsection{A suitable observability inequality} Let $\by\in \X_{\pw, T}^3(0, L)$, whose existence is guaranteed by \cref{prop:NLwp-Xs}. Let us introduce the operator $\mc{L}: \mc{V}\to L^2(Q)$ given by
\begin{align}\label{S4:def:L}
    \mc{L}z=z_t+p(x)z_{xxx}+z_x+(\overline{y}z)_x,
\end{align}
defined on the space of functions
\begin{align*}
    \mc{V}=\{z\in L^2(0, T; H_{\pw}^3(0, L))\ |\ \mc{L}z\in L^2(Q),\ z(0)=z(L)=z'(L)=0 \text{ and } z \text{ satisfies } \eqref{eq:tc}\}.
\end{align*}
In what follows, let $\omega_0\Subset\omega$ be non-empty and open, $\kappa\in (1, 2)$ and $\beta$ be constructed by \cref{S3:lem:weightObs}. From now on, let us fix $\ld\geq \ld_0$ large enough so the Carleman estimate of \cref{thm:carleman_omega} holds true with the weights $\eta$ and $\xi$ introduced in \eqref{S1:intro:weights}. Let us denote
\begin{align}\label{S4:def:weights2}
    \weta(t)=\max_{x\in [0, L]}\eta(t, x),\ \ \ueta(t)=\min_{x\in [0, L]}\eta(t, x),\ \ \zeta(t)=\dfrac{1}{t^2(T-t)^2},
\end{align}
We have the following one-parameter Carleman estimate.

\begin{proposition}\label{S4:prop:carleman-I}
    Let $(\omega, p)$ satisfy Hypothesis \ref{assumM}. Let $\by\in \X_{\pw, T}^3(0, L)$ be a solution of \eqref{eq:NLKdV-traj}. There exist $s_0>0$ and $C>0$ depending on $\omega$, $\Gamma$, $L$, $T$, $p$, $s_0$, $\ld_0$ and $\norm{\beta}_{C^3([0, L]\setminus\Gamma)}$ such that for any $\vp_T\in \D(\A^*)$ and $g\in L^2(0, T; \D(\A^*))$, the corresponding solution $\vp$ to \eqref{eq:KdV-lin-adj} satisfies
    \begin{multline}\label{S4:prop:ineq:carleman-I}
        \iint_Q e^{-4s\weta}\big(s^5\zeta^5|\vp|^2+s^3\zeta^3|\vp_x|^2+s\zeta|\vp_{xx}|^2\big)dxdt\\\leq C\left(\iint_Q e^{-2s\weta}|g|^2dxdt+s^7\iint_{(0, T)\times\omega}e^{-6s\ueta+2s\weta} \zeta^7|\vp|^2dxdt\right),
    \end{multline}
    for any $s\geq s_0$.
\end{proposition}

\begin{proof} The proof is made in two steps: we first decompose the solution of \eqref{eq:KdV-lin-adj} to have a regular source term with suitable decay in $t$ and for which we will apply the Carleman estimate. This decomposition will then allow us to employ a bootstrap argument to estimate the local terms coming from the higher order norms on the right-hand side of the Carleman estimate.

\medskip
\noindent{\emph{Step 1. Decomposition of the solution}.} Let us decompose the solution $\vp$ of \eqref{eq:KdV-lin-adj}, with the aim of obtaining $L^2$ regularity on the right-hand side of \eqref{eq:KdV-lin-adj}. Let us introduce $z$ and $u$, solutions of
\begin{align}\label{eq:KdV-lin-adj-d1}
\left\{
\begin{array}{rll}
-z_t-p(x)z_{xxx}-z_x-\by z_x = \rho_0g, & ~(t, x)\in (0, T)\times (0, L),\\
z(t, 0)=z(t, L)=z_{x}(t, 0)=0, & ~t\in (0, T),\\
z(T, x) =  0, & ~x\in (0, L),
\end{array}
\right.
\end{align}
and
\begin{align}\label{eq:KdV-lin-adj-d2}
\left\{
\begin{array}{rll}
-u_t-p(x)u_{xxx}-u_x-\by u_x = (-\rho_0)_t\vp, & ~(t, x)\in (0, T)\times (0, L),\\
u(t, 0)=u(t, L)=u_{x}(t, 0)=0, & ~t\in (0, T),\\
u(T, x) =  0, & ~x\in (0, L),
\end{array}
\right.
\end{align}
both of them coupled by the corresponding transmission conditions \eqref{eq:tc}, with $\rho_0(t):=e^{-s\weta}$. By uniqueness, we have $\rho_0\vp=u+z$. For the first system, using the regularity result for \eqref{eq:KdV-lin-adj-d1}, we have
\begin{align}\label{eq:z-H2-estimate}
    \norm{z}_{L^2(0, T; H_{\pw}^2(0, L))}^2\leq C\norm{\rho_0g}_{L^2(Q)}^2.
\end{align}
Now, we apply the Carleman estimate of \cref{thm:carleman_omega} for solutions of \eqref{eq:KdV-lin-adj-d2}, with the new weights \eqref{S4:def:weights2} and fixed $\ld$ as above, obtaining
\begin{multline}\label{ineq:carleman-1p}
C\iint_Q e^{-2s\weta}\big(s^5\zeta^5|u|^2+s^3\zeta^3|u_x|^2+s\zeta|u_{xx}|^2\big)dxdt\leq \iint_Q e^{-2s\weta}|g|^2dxdt\\+\iint_{(0, T)\times\omega}e^{-2s\ueta} \big(s^5\zeta^5|u|^2dxdt+s^3\zeta^3|u_x|^2dxdt+s \zeta|u_{xx}|^2\big)dxdt.
\end{multline}
Observe that on the right-hand side, we used $|(\rho_0)_t\vp|\leq Cs\xi^{3/2}|\rho_0\vp|$ followed by the relation $\rho_0\vp=u+z$, to then absorb the term containing $u$ for $s$ large enough and conclude using estimate \eqref{eq:z-H2-estimate}.

\medskip
\noindent{\emph{Step 2. Local estimates}.} By classical interpolation (see Lions-Magenes \cite[Section 9]{LM72})
\begin{align*}
    H^1(\omega)=(L^2(\omega), H^3(\omega))_{1/3, 2}\ \text{ and }\ H^2(\omega)=(L^2(\omega),  H^3(\omega))_{2/3, 2}.
\end{align*}
Let $\veps>0$. By Young's inequality with $(p, q)=(3/2, 3)$ we get
\begin{multline*}
    s^3\iint_{(0, T)\times\omega}\zeta^3e^{-2s\ueta}|u_x|^2dxdt\leq s^3\int_0^T \zeta^3e^{-2s\ueta}\norm{u}_{L^2(\omega)}^{4/3}\norm{u}_{H^3(\omega)}^{2/3}dt\\
    \leq C_\veps s^{11/2}\int_0^T \zeta^{11/2}e^{-3s\ueta+s\weta}\norm{u}_{L^2(\omega)}^2dt+\veps s^{-2}\int_0^T \zeta^{-2}e^{-s2\weta}\norm{u}_{H^3(\omega)}^2dt.
\end{multline*}
Similarly, with $(p, q)=(3, 3/2)$ we get
\begin{multline*}
    s\iint_{(0, T)\times\omega}e^{-2s\ueta}\zeta|u_{xx}|^2dxdt\leq s\int_0^T \zeta e^{-2s\ueta}\norm{u}_{L^2(\omega)}^{2/3}\norm{u}_{H^3(\omega)}^{4/3}dt\\
    \leq C_\veps s^{7}\int_0^T \zeta^{7}e^{-6s\ueta+4s\weta}\norm{u}_{L^2(\omega)}^2dt+\veps s^{-2}\int_0^T \zeta^{-2}e^{-2s\weta}\norm{u}_{H^3(\omega)}^2dt.
\end{multline*}
From the Carleman estimate and the inequalities above, we get
\begin{multline*}
    C\iint_Q e^{-2s\eta}\big(s^5\zeta^5|u|^2+s^3\zeta^3|u_x|^2+s\zeta|u_{xx}|^2\big)dxdt\\ \leq \iint_Q e^{-2s\weta}|g|^2dxdt+s^7\iint_{(0, T)\times\omega}\zeta^7e^{-6s\ueta+4s\weta}|u|^2dxdt+\veps\left(s^{-2}\int_0^T \zeta^{-2}e^{-2s\weta}\norm{u}_{H^3(\omega)}^2dt\right).
\end{multline*}
We want to estimate the local term containing $\norm{u}_{H^3(\omega)}^2$. By looking at the weights accompanying the local $H^3$-norm, let us introduce $\widehat{u}=\widehat{\rho}(t)u$ with $\widehat{\rho}(t):=s^{-1/2}\ld^{-1}\zeta^{-1/2} e^{-s\weta}$. We thus see that $\widehat{u}$ solves \eqref{eq:KdV-lin-adj-d3} with $\Tilde{\rho}$ replaced by $\widehat{\rho}$,
\begin{align}\label{eq:KdV-lin-adj-d4}
\left\{
\begin{array}{rll}
-\widehat{u}_t-p(x)\widehat{u}_{xxx}-\widehat{u}_x-\by \widehat{u}_x = \widehat{\rho}(-\rho_0)_t\vp-\widehat{\rho}_t u, & ~(t, x)\in (0, T)\times (0, L),\\
\widehat{u}(t, 0)=\widehat{u}(t, L)=\widehat{u}_{x}(t, 0)=0, & ~t\in (0, T),\\
\widehat{u}(T, x) =  0, & ~x\in (0, L),
\end{array}
\right.
\end{align}
coupled by the corresponding transmission conditions \eqref{eq:tc}. Since $\vp\in C([0, T], \D(\A^*))$, using the regularity estimates given by \cref{prop:wp:adj-linearized}, we have
\begin{align*}
    \norm{s^{-1/2}\zeta^{-1/2} e^{-s\weta}u}_{L^2(0, T; H_{\pw}^3(0, L))}^2&=\norm{\widehat{u}}_{L^2(0, T; H_{\pw}^3(0, L))}^2\\ &\leq C\big(\norm{s^{1/2}\zeta e^{-s\weta}u}_{L^2(0, T; H^1(0, L))}^2+\norm{s^{1/2}\zeta e^{-2s\weta}\vp}_{L^2(0, T; H^1(0, L))}^2\big).
\end{align*}
We are then led to define $\Tilde{u}=\Tilde{\rho}(t)u$ with $\Tilde{\rho}(t)=s^{1/2}\zeta e^{-s\weta}$ aiming to estimate the first term of the right-hand side in the inequality above. We see that $\Tilde{u}$ is the solution of
\begin{align}\label{eq:KdV-lin-adj-d3}
\left\{
\begin{array}{rl}
-\Tilde{u}_t-p(x)\Tilde{u}_{xxx}-\Tilde{u}_x-\by \Tilde{u}_x = \Tilde{\rho}(-\rho_0)_t\vp-\Tilde{\rho}_t u, & ~(t, x)\in (0, T)\times (0, L),\\
\Tilde{u}(t, 0)=\Tilde{u}(t, L)=\Tilde{u}_{x}(t, 0)=0, & ~t\in (0, T),\\
\Tilde{u}(T, x) =  0, & ~x\in (0, L),
\end{array}
\right.
\end{align}
coupled by the corresponding transmission conditions \eqref{eq:tc}. As $|(\rho_0)_t|\lesssim s\zeta^{3/2}e^{-s\weta}$, we get
\begin{align*}
    |\Tilde{\rho}_t|=s^{1/2}|\zeta_t\rho_0(t)+\zeta(\rho_0)_t|\lesssim \big(s^{1/2}\zeta^{3/2}+s^{3/2}\zeta^{5/2}\big)\rho_0\lesssim s^{3/2}\zeta^{5/2}e^{-s\weta}.
\end{align*}
Using once again \cref{prop:wp:adj-linearized}, we obtain
\begin{align*}
    \norm{s^{1/2}\zeta e^{-s\weta}u}_{L^2(0, T; H_{\pw}^2(0, L))}^2=\norm{\Tilde{u}}_{L^2(0, T; H_{\pw}^2(0, L))}^2\leq C\big(\norm{s^{3/2}\zeta^{5/2} e^{-s\weta}u}_{L^2(Q)}^2+\norm{s^{3/2}\zeta^{5/2}e^{-2s\weta}\vp}_{L^2(Q)}^2\big).
\end{align*}
Gathering the above inequalities, we have
\begin{multline*}
    \norm{s^{-1/2}\zeta^{-1/2} e^{-s\weta}u}_{L^2(0, T; H_{\pw}^3(0, L))}^2\leq C\Big(\norm{s^{3/2}\zeta^{5/2} e^{-s\weta}u}_{L^2(Q)}^2\\+\norm{s^{3/2}\zeta^{5/2}e^{-2s\weta}\vp}_{L^2(Q)}^2+\norm{s^{1/2}\zeta e^{-2s\weta}\vp}_{L^2(0, T; H^1(0, L))}^2\Big).
\end{multline*}
Since $(s, t)\mapsto s^{3/2}\zeta^{5/2}e^{-s\weta}$ is bounded, using \eqref{eq:z-H2-estimate} we see that the right-hand side of the inequality above is bounded by the left-hand side of the Carleman estimate \eqref{ineq:carleman-1p} and $\norm{\rho_0 g}_{L^2(Q)}^2$. Therefore
\begin{multline*}
    \iint_Q e^{-2s\eta}\big(s^5\zeta^5|u|^2+s^3\zeta^3|u_x|^2+s\zeta|u_{xx}|^2\big)dxdt+\norm{s^{-1/2}\zeta^{-1/2} e^{-s\weta}u}_{L^2(0, T; H_{\pw}^3(0, L))}^2\\ \leq C\left(\iint_Q e^{-2s\weta}|g|^2dxdt+s^7\iint_{(0, T)\times\omega}e^{-6s\ueta+4s\weta}\zeta^7|u|^2dxdt\right)+\veps\left(s^{-2}\int_0^T \zeta^{-2}e^{-2s\weta}\norm{u}_{H^3(\omega)}^2dt\right).
\end{multline*}
Choosing $\veps>0$ small enough, the last term on the right-hand side above, can be absorbed by the last term on the left-hand side above. To return to the $\vp$ variable, we use $\rho_0\vp=z+u$ and estimate \eqref{eq:z-H2-estimate} to get
\begin{multline*}
    \iint_Q e^{-4s\weta}\big(s^5\xi^5|\vp|^2+s^3\xi^3|\vp_x|^2+s\xi|\vp_{xx}|^2\big)dxdt\\
    \leq C\left(\iint_Q e^{-2s\weta}|g|^2dxdt+\iint_Q e^{-2s\eta}\big(s^5\zeta^5|u|^2+s^3\zeta^3|u_x|^2+s\zeta|u_{xx}|^2\big)dxdt\right).
\end{multline*}
Once again, using estimate \eqref{eq:z-H2-estimate} and that $(s, t)\mapsto s^7e^{-6s\ueta+4s\weta}\zeta^7$ is bounded for $s\geq s_0$ and $t\in (0, T)$, we obtain
\begin{align*}
    s^7\iint_{(0, T)\times\omega}e^{-6s\ueta+4s\weta}\zeta^7|u|^2dxdt&\leq C\left(\iint_Q e^{-2s\weta}|g|^2dxdt+s^7\iint_{(0, T)\times\omega}e^{-6s\ueta+2s\weta}\zeta^7|\vp|^2dxdt\right).
\end{align*}
Putting together the three last estimates, we arrive to inequality \eqref{S4:prop:ineq:carleman-I}, finishing the proof.
\end{proof}

For notational convenience, we introduce $\mc{L}^*: \mc{V}^*\to L^2(Q)$, the adjoint operator to \eqref{S4:def:L},
\begin{align}\label{S4:def:L*}
    \mc{L}^*\psi:=-\psi_t-p(x)\psi_{xxx}-\psi_x-\overline{y}\psi_x,
\end{align}
acting on the space of functions
\begin{align*}
    \mc{V}^*:=\{\psi\in L^2(0, T; H_{\pw}^3(0, L))\ |\ \mc{L}^*\psi\in L^2(Q),\ \psi(0)=\psi(L)=\psi'(0)=0 \text{ and } \psi \text{ satisfies } \eqref{eq:tc}\}.
\end{align*}
We now introduce weight that does not vanish at $t=0$. Let $\ell\in C^1([0, T])$ be a positive function in $[0, T)$ defined by
\begin{align*}
    \ell(t)=\left\{\begin{array}{cc}
        T^2/4     &\ t\in [0, T/2], \\
        t(T-t)    &\ t\in [T/2, T].
    \end{array}\right.
\end{align*}
We then consider
\begin{align*}
    \al(t, x)=(e^{\kappa \ld \norm{\beta}_\infty}-e^{\ld\beta(x)})\tau(t),\ \ \tau(t)=\dfrac{1}{\ell(t)},\ \
    \wal(t)=\max_{x\in [0, L]}\al(t, x),\ \ \ual(t)=\min_{x\in [0, L]}\al(t, x).    
\end{align*}
We further ask that $\lambda\geq \kappa^2/\norm{\beta}_\infty$, where $\kappa$ is the parameter  used in \cref{S3:lem:weightObs}. Thus, from now on we assume that $\ld\geq \max\{\ld_0, \kappa^2/\norm{\beta}_\infty\}$ and therefore $2\wal<3\ual$ holds. As a consequence of \cref{S4:prop:carleman-I} we have the following weighted observability inequality.

\begin{lemma}\label{S4:lem:carleman-II}
    Under the assumptions of \cref{S4:prop:carleman-I}, there exist $s$ and $C$ such that every solution $\vp$ of \eqref{eq:KdV-lin-adj} satisfies
    \begin{multline}\label{S4:lem:ineq:carleman_decay}
        \iint_Q e^{-4s\wal}\big(\tau^5|\vp|^2+\tau^3|\vp_x|^2+\tau|\vp_{xx}|^2\big)dxdt+\norm{\vp(0)}_{L^2(0, L)}^2\\ \leq C\left(\iint_Q e^{-2s\wal}|g|^2dxdt+\iint_{(0, T)\times\omega}e^{-6s\ual+2s\wal} \tau^7|\vp|^2dxdt\right).
    \end{multline}
\end{lemma}
\begin{proof}
    By construction, $\eta=\al$ and $\tau=\zeta$ in $(T/2, T)\times [0, L]$. Therefore, as a consequence of \cref{S4:prop:carleman-I} we get
    \begin{multline*}
        \int_{T/2}^T\int_0^L e^{-4s\wal}\big(s^5\tau^5|\vp|^2+s^3\tau^3|\vp_x|^2+s\tau|\vp_{xx}|^2\big)dxdt\\
        =\int_{T/2}^T\int_0^L e^{-4s\weta}\big(s^5\zeta^5|\vp|^2+s^3\zeta^3|\vp_x|^2+s\zeta|\vp_{xx}|^2\big)dxdt\\
        \leq C\left(\iint_Q e^{-2s\weta}|g|^2dxdt+s^7\iint_{(0, T)\times\omega}e^{-6s\ueta+4s\weta} \zeta^7|\vp|^2dxdt\right).
    \end{multline*}
    From now on, let us fix $s\geq s_0$. By construction of the weights, we only need to focus the analysis on $(0, T/2)$. Using inequalities $e^{-2s\weta}\leq C$ and $e^{-6s\ueta+4s\weta} \zeta^7\geq C$ in $[0, T/2]$, followed by the fact that $\tau$ is constant in $[0, T/2]$, we get
    \begin{multline}\label{ineq:carleman-aux1}
        \int_{T/2}^T\int_0^L e^{-4s\wal}\big(s^5\tau^5|\vp|^2+s^3\tau^3|\vp_x|^2+s\tau|\vp_{xx}|^2\big)dxdt\\
        \leq C\left(\iint_Q e^{-2s\wal}|g|^2dxdt+\iint_{(0, T)\times\omega}e^{-6s\ual+4s\wal} \tau^7|\vp|^2dxdt\right).
    \end{multline}
        
    Let us take a cutoff $\chi\in C^1([0,T])$ such that $\chi\equiv 1$ in 
	$[0, T/2]$ and $\chi\equiv 0$ in $[3T/4,T]$. Observe that $\chi\vp\in \mc{V}^*$, $\chi(T)\vp(T,\cdot)=0$ and $\mc{L}^*(\chi\vp)=\chi\mc{L}^*\vp-\chi'\vp$. Thus, given that $g\in L^2(0, T; \D(\A^*))$, by semigroup estimates we get
    \begin{align*}
        \norm{\chi\vp}_{C([0, T], L^2(0, L))}\leq C\norm{\chi g-\chi'\vp}_{L^2(0, T; L^2(0, L))},
    \end{align*}
    from which follows
    \begin{align*}
        \norm{\vp}_{C([0, T/2], L^2(0, L))}\leq C\norm{g}_{L^2(0, 3T/4; L^2(0, L))}+\norm{\vp}_{L^2(T/2, 3T/4; L^2(0, L))}.
    \end{align*}
    By employing \cref{prop:wp:adj-linearized} and the above estimate, we obtain
	\begin{align*}
		\norm{\vp(0)}_{L^2(0,L)}^2+\norm{\vp}_{L^2(0,T/2; H_{\pw}^2(0,L))}^2
		\leq C\big(\norm{g}_{L^2(0,3T/4; L^2(0,L))}^2+\norm{\vp}_{L^2(T/2,3T/4;L^2((0,L))}^2\big). 
	\end{align*}
    Taking into account that 
    \begin{align*}
        \tau^5e^{-2s\wal}\geq C>0,\ \forall t\in [T/2, 3T/4]\ \text{ and }\ e^{-4s\wal}\geq C>0,\ \forall t\in[0,3T/4],
    \end{align*}
    we arrive to
    \begin{multline}\label{ineq:carleman-aux2}
        \int_0^{T/2}\int_0^L e^{-4s\wal}\big(s^5\tau^5|\vp|^2+s^3\tau^3|\vp_x|^2+s\tau|\vp_{xx}|^2\big)dxdt+\norm{\vp(0)}_{L^2(0, L)}^2\\ \leq C\left(\int_0^{3T/4}\int_0^L e^{-2s\wal}|g|^2dxdt+\int_{T/2}^{3T/4}\int_0^Le^{-4s\wal} \tau^5|\vp|^2dxdt\right).
    \end{multline}
    Inequality \eqref{S4:lem:ineq:carleman_decay} then follows, upon adjusting $s\geq s_0$ if necessary, by combining \eqref{ineq:carleman-aux1} and \eqref{ineq:carleman-aux2}.
\end{proof}

\subsection{Null controllability of the linearized system}
Let us introduce the space
\begin{multline*}
    \mc{E}:=\{(z,v)\ |\ e^{s\wal}z\in L^2(Q),\ \tau^{-9/2}e^{3s\ual-s\wal}v\mathbbm{1}_{\omega}\in L^2(Q),\\
    	e^{s\wal}\tau^{-3/2}z\in \X_T^0(0, L),\ e^{2s\wal}\tau^{-5/2}(\mc{L}z-\mathbbm{1}_\omega v)\in L^2(0,T;H^{-1}(0,L))\},
\end{multline*}
which is a Banach space when equipped with the norm whose square is given by
\begin{multline*}
    \norm{(z, v)}_\mc{E}^2=\norm{e^{s\wal}z}_{L^2(Q)}^2+\norm{\tau^{-9/2}e^{3s\ual-s\wal}v\mathbbm{1}_{\omega}}_{L^2(Q)}^2\\
    +\norm{e^{s\wal}\tau^{-3/2}z}_{\X_T^0(0, L)}+\norm{e^{2s\wal}\tau^{-5/2}(\mc{L}z-v\mathbbm{1}_{\omega})}_{L^2(0,T;H^{-1}(0,L))}^2.
\end{multline*}
We now aim to solve \eqref{eq:KdV-lin-control} in the space $\mc{E}$ with a right-hand side in an appropriate weighted space. Indeed, in such case, from which the null controllability of the system follows given that $e^{s\wal}\tau^{-3/2}z\in C([0, T], L^2(0, L))$ implies that $z(T,\cdot )=0$.

\begin{proposition}\label{S4:prop:lincontrol}
    Let $(\omega, p)$ satisfy Hypothesis \ref{assumM} and let $T>0$. For any $z_0\in L^2(0,L)$ and $e^{2s\wal}\tau^{-5/2}h\in L^2(Q)$, there exists a function $v\in L^2(0,T; L^2(\omega))$ such that the associated solution $(z,v)$ to \eqref{eq:KdV-lin-control} satisfies $(z,v)\in \mc{E}$. Furthermore, there exists $C>0$ such that
    \begin{align}\label{prop:eq:lincontrol-bound}
        \norm{v}_{L^2(0, T; L^2(\omega))}\leq C\big(\norm{z_0}_{L^2(0, L)}+\norm{h}_{L^2(Q)}\big).
    \end{align}
\end{proposition}
\begin{proof}
    Set $\mc{Q}_0$ to be the space of functions $\vp\in C^3([0, T]\times([0, L]\setminus \Gamma))$ such that:
    \begin{itemize}
        \item $\vp_{|_{I_k}}\in C^3([0, T]\times \overline{I}_k)$, $k\in\ran{0}{N-1}$;
        \item $\vp$ satisfies the transmission conditions \eqref{eq:tc};
        \item $\vp$ satisfies the boundary conditions $\vp(t, 0)=\vp(t, L)=\vp_x(t, 0)=0$, $t\in (0, T)$.
\end{itemize}
    Let us introduce the bilinear form $\boldsymbol{a}(\cdot, \cdot)$ on $\mc{Q}_0$
    \begin{align*}
        \boldsymbol{a}(\hat\vp,w):=\iint_{Q}e^{-2s\wal}(\mc{L}^*\hat\vp)(\mc{L}^*w)dxdt
		+\iint_{\omega\times(0,T)}e^{-6s\ual+2s\wal}\tau^7\hat\vp wdxdt,
		\quad \forall (\hat\vp, w)\in \mc{Q}_0\times \mc{Q}_0,
    \end{align*}
    where $\mc{L}^*$ is the adjoint operator of $\mc{L}$ defined in \eqref{S4:def:L}. Observe that Carleman inequality \eqref{S4:lem:ineq:carleman_decay} is applicable for any $w\in \mc{Q}_0$, thus
    \begin{align}\label{ineq:bilcoercive}
        \iint_Q \tau^5 e^{-4s\wal}|w|^2dxdt+\norm{w(0)}_{L^2(0, L)}^2\leq C\boldsymbol{a}(w, w),\ \forall w\in \mc{Q}_0.
    \end{align}
    In particular, a unique continuation property holds, in other words, $\boldsymbol{a}(w, w)=0$ implies that $w=0$ in $\mc{Q}_0$. Further, observe that the bound given here above implies that  $\boldsymbol{a}(\cdot, \cdot): \mc{Q}_0\times \mc{Q}_0\to \R$ is a coercive bilinear form. It being symmetric as well, $\boldsymbol{a}(\cdot, \cdot)$ defines an inner product in $\mc{Q}_0$. We introduce $\mc{Q}$ as the completion of $\mc{Q}_0$ for the form induced by $\boldsymbol{a}(\cdot, \cdot)$, which we denote by $\norm{\cdot}_\mc{Q}$. Certainly, $\mc{Q}$ is a Hilbert space and $\boldsymbol{a}(\cdot,\cdot)$ is a continuous and coercive bilinear form on $\mc{Q}$.

    Let us introduce the linear form $\G$, given by
    \begin{align*}
        \inn{\G, w}:=\iint_Q hwdxdt+\int_0^L z_0(x)w(0, x)dx,\ \forall w\in \mc{Q}.
    \end{align*}
    Given that $e^{2s\wal}\tau^{-5/2}h\in L^2(Q)$, by the Carleman inequality \eqref{S4:lem:ineq:carleman_decay}, the linear form $w\in \mc{Q}\mapsto \inn{G, w}\in \R$ is well-defined and continuous. Indeed,
    \begin{align*}
        |\inn{\G, w}|\leq \norm{e^{s\wal}\tau^{-5/4}h}_{L^2(Q)}\norm{e^{-s\wal}\tau^{5/4}w}_{L^2(Q)}+\norm{z_0}_{L^2(0, L)}\norm{w(0, \cdot)}_{L^2(0, L)},
    \end{align*}
    and using inequality \eqref{ineq:bilcoercive} along with the density of $\mc{Q}_0$ in $\mc{Q}$, we have
    \begin{align}\label{ineq:lin-continuity}
        |\inn{\G, w}|\leq \big(\norm{e^{s\wal}\tau^{-5/4}h}_{L^2(Q)}+\norm{z_0}_{L^2(0, L)}\big)\norm{w}_{\mc{Q}},
    \end{align}
    valid for any $w\in \mc{Q}$. Applying Lax-Milgram's lemma, there exists a unique $\hat\vp\in \mc{Q}$ such that
    \begin{align}\label{eq:bilinear-weak-sol}
        \boldsymbol{a}(\hat\vp, w)=\inn{\G, w},\ \forall w\in \mc{Q}.
    \end{align}

    Introduce
    \begin{align*}
        \left\{\begin{array}{ll}
            \hat{z}:=e^{-2s\wal}\mc{L}^*\hat\vp,    & \text{ in } Q,  \\
            \hat{v}:=-e^{-6s\ual+2s\wal}\tau^7\hat\vp,     & \text{ in } (0, T)\times \omega.  
        \end{array}\right.
    \end{align*}
    Given that $\hat{\vp}\in \mc{Q}$, we notice that the pair $(\hat{z}, \hat{v})$ verifies
    \begin{align}\label{S4:eq:bilinear-haty}
        \boldsymbol{a}(\hat\vp, \hat\vp)=\iint_Q e^{2s\wal}|\hat{z}|^2dxdt+\iint_{(0, T)\times\omega} e^{6s\ual-2s\wal}\tau^{-7}|\hat{v}|^2dxdt<+\infty.
    \end{align}
    Furthermore, we see that $\hat{z}$ is the unique solution by transposition of \eqref{eq:KdV-lin-control} with $v$ replaced by $\hat{v}$. Indeed, from \eqref{eq:bilinear-weak-sol} we readily get the variational identity: for every $g\in L^2(Q)$ we have
    \begin{align*}
        \iint_Q \hat{z}gdxdt=\iint_Q (h+\hat{v})wdxdt+\int_0^L z_0(x)w(0, x)dx,
    \end{align*}
    with $w\in \X_T^0(0, L)$ solution of the adjoint \eqref{eq:KdV-lin-adj} with right-hand side $g$ and $w(T, \cdot)=0$, whose existence is guaranteed by \cref{prop:wp:adj-linearized}.

    As a last step, we verify that $(\hat{z}, \hat{v})\in \mc{E}$. From \eqref{S4:eq:bilinear-haty} we readily get $e^{s\wal}\hat{z}\in L^2(Q)$ and $e^{3s\ual-s\wal}\tau^{-7/2}\hat{v}\in L^2(Q)$. Moreover, using the equation and that $e^{2s\wal}\tau^{-5/2}h\in L^2(Q)$, we readily get
    \begin{align*}
        e^{2s\wal}\tau^{-5/2}(\mc{L}\hat{z}-\mathbbm{1}_\omega\hat{v})\in L^2(Q).
    \end{align*}
    To check that $e^{s\wal}\tau^{-3/2}\hat{z}\in \X_T^0(0, L)$, we define
    \begin{align*}
        z^*:=e^{s\wal}\tau^{-3/2}\hat{z}\ \text{ and }\ h^*=e^{s\wal}\tau^{-3/2}(h+\hat{v}).
    \end{align*}
    Observe that $z^*$ satisfies the system
    \begin{align*}
        \left\{
        \begin{array}{rll}
        z^*_t+p(x)z^*_{xxx}+z^*_x+(\by z^*)_x = h^* + (e^{s\wal}\tau^{-3/2})_t\hat{z} , & ~(t, x)\in (0, T)\times (0, L),\\
        z^*(t, 0)=z^*(t, L)=z^*_{x}(t, L)=0, & ~t\in (0, T),\\
        z^*(0, x) = e^{s\wal(0)}\tau^{-3/2}(0)\hat{z}_0(x), & ~x\in (0, L),
        \end{array}
        \right.
    \end{align*}
    coupled by the corresponding transmission conditions \eqref{eq:tc}. Since $e^{s\wal}h\in L^2(Q)$ and $2\wal<3\ual$, we get $h^*\in L^2(Q)$ and $(e^{s\wal}\tau^{-3/2})_t\hat{z}\in L^2(Q)$. For $\hat{z}_0\in L^2(0, L)$, \cref{prop:wp-kdv-f-1} along with an argument similar to the one used in \cref{prop:wp:adj-linearized} give us $z^*\in \X_T^0(0, L)$.

    By considering $\hat{v}$ as before, the bilinear form $\boldsymbol{a}$ and identity \eqref{eq:bilinear-weak-sol}, we obtain estimate \eqref{prop:eq:lincontrol-bound}.
\end{proof}

\subsection{Control of the nonlinear system} The last step relies on a local inversion result.

\begin{theorem}\label{S4:thm:inversemap}\cite[Chapter I, Section 4, Theorem 4.1]{FI96}
	Suppose that $\mathcal{B}_1$, $\mathcal{B}_2$ are Banach spaces and $\mathcal{F}:\mathcal{B}_1\to \mathcal{B}_2$ is a continuously differentiable map. We assume that for $b_1^0\in \mathcal{B}_1$, $b_2^0\in \mathcal{B}_2$ the equality
	\begin{align*}
		\mathcal{F}(b_1^0)=b_2^0
	\end{align*}
	holds and $\mathcal{F}'(b_1^0):\mathcal{B}_1\to \mathcal{B}_2$ is a surjective. Then there exists $\delta >0$ such that for any $b_2\in \mathcal{B}_2$ which satisfies the condition $\norm{b_2^0-b_2}_{\mathcal{B}_2}<\delta$ there exists a solution $b_1\in \mathcal{B}_1$ of the equation
	\begin{equation*}
		\mathcal{F}(b_1)=b_2.
	\end{equation*}
\end{theorem}

We now prove the main control result for the nonlinear system.

\begin{proof}[Proof of \cref{thm:NLKdV-control}] Let us set
\begin{align*}
    \B_1:=\mc{E}\ \text{ and }\ \B_2=L^2\big(e^{2s\wal}\tau^{-5/2}(0, T); L^2(0, L)\big)\times L^2(0, L)
\end{align*}
and the operator $\F:\B_1\to \B_2$ defined by
\begin{align*}
    \F(y, v)=\big(z_t+p(x)z_{xxx}+z_x+(\by z)_x+zz_x-\mathbbm{1}_\omega v, z(0)\big).
\end{align*}
We now prove that $\F$ is of class $C^1(\B_1, \B_2)$. Let us assume that $\by\in \X_T^0(0, L)$. By linearity, it only remains to prove that the bilinear operator
\begin{align*}
    \big((z^1, v^1), (z^2, v^2)\big)\in \mc{E}\times \mc{E}\longmapsto \dfrac{1}{2}(z^1z^2)_x\in L^2\big(e^{2s\wal}\tau^{-5/2}(0, T); L^2(0, L)\big)
\end{align*}
is continuous. Observe that
\begin{align*}
    e^{2s\wal}\tau^{-5/2}z\in \X_T^0(0, L),
\end{align*}
for any $(z, v)\in \mc{E}$. By Sobolev embedding $H^1(0, L)\hookrightarrow L^\infty(0, L)$, we have
\begin{align*}
    \norm{e^{2s\wal}\tau^{-5/2}(z^1z^2)_x}_{L^2(Q)}&\leq C\int_0^T \big(e^{2s\wal}\tau^{-3}\norm{z^1}_{L^\infty(0, L)}^2e^{2s\wal}\tau^{-3}\norm{z^2}_{H^1(0, L)}+\\
    &\hspace{3cm}+e^{2s\wal}\tau^{-3}\norm{z^2}_{L^\infty(0, L)}^2e^{2s\wal}\tau^{-3}\norm{z^1}_{H^1(0, L)}\big)dt\\
    &\leq C\norm{z^1}_{\B_1}\norm{z^2}_{\B_1}.
\end{align*}
We are in position to apply \cref{S4:thm:inversemap}, with $b_1^0=(0, 0)\in \B_1$ and $b_2=0\in \B_2$. The derivative $\F'(0, 0):\B_1\to \B_2$ is given by
\begin{align*}
    \F'(0, 0)(z, v)=\big(z_t+p(x)z_{xxx}+z_x+(\by z)_x-\mathbbm{1}_\omega v, z(0)\big),\ \forall (z, v)\in \B_1.
\end{align*}
Thus, there exists $\delta>0$ such that, if $\norm{z(0)}_{L^2(0, L)}\leq \delta$, we can find a control $v$ such that the associated solution $z$ of the nonlinear system \eqref{eq:NLKdV-control-dif} satisfies $z(T, \cdot)=0$ on $(0, L)$. This finishes the proof.
\end{proof}


\section{Lipschitz stability in retrieving an unknown potential}\label{S5}

In this section we follow \cite{BCCM14}. A key point in the latter work is that some symmetry assumptions on the coefficient to recover and on the initial data are made, in order to avoid an observation of the solution in some time $T_0>0$, as usual in the parabolic case. To adapt this point to our case, we introduced Assumption \ref{assumG}, which will allow us to apply the Carleman estimate and carry out the method.

We will need the following slight modification of \cref{thm:carleman_omega}. Let $\omega_0\Subset \omega$ be non-empty and open, $\kappa\in (1, 2)$ and $\beta$ be constructed by \cref{S3:lem:weightObs} with $\omega_0$ as before. A Carleman estimate on $Q:=(-T, T)\times (0, L)$ like the one in \cref{thm:carleman_omega} can be derived just by modifying the weights $\eta$ and $\xi$ as follows:
\begin{align*}
        \eta(t, x)=\dfrac{e^{\kappa\ld\norm{\beta}_{\infty}}-e^{\ld\beta(x)}}{(t+T)(T-t)}\ \text{   and   }\   \xi(t, x)=\dfrac{e^{\ld\beta(x)}}{(t+T)(T-t)}
\end{align*}
for $(t, x)\in Q$.
More precisely, we have the following.
\begin{proposition}\label{S5:prop:carlemanomega}
    Let $\eta$ and $\xi$ be as previously defined. Under Hypothesis \ref{assumM}, there exist $s_0>0$, $\ld_0>0$ and a constant $C>0$ depending on $\omega$, $\Gamma$, $L$, $T$, $p$, $\norm{\beta}_{C^3([0, L]\setminus\Gamma)}$, $s_0$ and $\ld_0$ such that for any $u\in \mc{V}$ we have
    \begin{multline}\label{S5:prop:ineq:carlemanomega}
        C\iint_{Q} e^{-2s\eta}(s^5\ld^6\xi^5|u|^2+s^3\ld^4\xi^3|u_x|^2+s\ld^2\xi|u_{xx}|^2)dxdt\leq \norm{e^{-s\eta}\mc{L}u}_{L^2(Q)}^2\\+\iint_{(-T, T)\times\omega}e^{-2s\eta} (s^5\ld^6 \xi^5|u|^2dxdt+s^3\ld^4\xi^3|u_x|^2dxdt+s\ld^2 \xi|u_{xx}|^2)dxdt
    \end{multline}
    for any $s\geq s_0$ and $\ld\geq \ld_0,$ with $\mc{L}$ and $\mc{V}$ similarly defined as in \eqref{S1:def:opL}.  
\end{proposition}

With this at hand, we can prove the Lipschitz stability result.

\begin{proof}[Proof of \cref{S1:thm:IPpotential}]
By \cref{prop:NLwp-Xs} and \cref{rk:BKadmissible}, for any $\nu=\mathfrak{P}_{\leq m}^{adm}$ the data $(y_0, \vec{h})\in\mc{Z}_{6, T}$ is $6$-compatible with respect to $v$, so both $y:=y[\mu]$ and $z:=z[\nu]$ belong to $\X_{\pw, T}^6(0, L)$. By Sobolev embedding on each $I_k$, this ensures the regularity needed to employ the Bukhge{\u\i}m-Klibanov method. Let us define
\begin{align*}
    u(t, x):=y(t, x)-z(t, x)\ \text{ and }\ \sigma(x):=\nu(x)-\mu(x).
\end{align*}

\noindent{\emph{Step 1: auxiliary system.}} Let $v=u_t$ and note that $v(0, x):=\sigma(x)y_0'(x)$ satisfies $v(0, x)=v(0, L-x)$ for every $x\in [0, L]$. Using the symmetry hypotheses, $\psi:=\widehat{v}$ satisfies
\begin{align*}
\left\{
\begin{array}{rll}
\psi_t+p(x)\psi_{xxx}+(1+\w{z})\psi_x+\w{y}_x\psi=\check{f}, & ~(t, x)\in (-T, T)\times (0, L),\\
\psi(t, 0)=\psi(t, L)=0, & ~t\in (-T, T),\\
\psi_{x}(t, L)=0, & ~t\in (0, T),\\
\psi_{x}(t, L)=-v_x(0, -t), & ~t\in (-T, 0),\\
\psi(0, x)=\sigma(x)y_{0}'(x), & ~x\in (0, L),
\end{array}\right.
\end{align*}
coupled by the corresponding transmission conditions \eqref{eq:tc} with coefficient $p$, with $f:=\sigma(x)z_{xt}-y_{xt}u-z_tu_x$ and the symmetric and anti-symmetric extensions being defined, respectively, as
\begin{align*}
    \widehat{g}(t, x)=\left\{\begin{array}{ll}
         g(t, x),& (t, x)\in [0, T]\times [0, L], \\
         g(t, L-x),& (t, x)\in [-T, 0)\times [0, L],
    \end{array}\right.
\end{align*}
\begin{align*}
    \check{g}(t, x)=\left\{\begin{array}{ll}
         g(t, x),& (t, x)\in [0, T]\times [0, L], \\
         -g(-t, L-x),& (t, x)\in [-T, 0)\times [0, L].
    \end{array}\right.
\end{align*}

\noindent{\emph{Step 2: First use of the Carleman estimate.}} By compactness, we can find $\omega_0\Subset \omega$ which is symmetric with respect to $L/2$ and $(\omega_0, p)$ satisfies Hypothesis \ref{assumM}. Let $K>0$ be some constant such that
\begin{align}\label{S5:ineq:Kbound}
    \max\{\norm{y}_{W^{1,\infty}(0, T;W^{1,\infty}(0, L))}, \norm{z}_{W^{1,\infty}(0, T;W^{1,\infty}(0, L))}\}\leq K.
\end{align}
This is consistent given that $y$, $z\in C([0, T], H_{\pw}^6(0, L))$ and by classical Sobolev embedding $H^1(I_k)\hookrightarrow L^\infty(I_k)$ applied on each $k\in\ran{0}{N-1}$.
We shall focus on the following integral
\begin{align}\label{S5:eq:BKidentity}
    \int_{-T}^0\int_0^L \xi^2 w\mc{L}_1 wdxdt&=\dfrac{1}{2}\int_0^L \xi^2(0, x)|w(0, x)|^2dx+\J,
\end{align}
where $\J$ can be estimated by the Carleman estimate for the conjugated operator (obtained throughout the proof of \cref{S5:prop:carlemanomega}, compare with \cref{S3:prop:carleman-conj}) as follows
\begin{multline}\label{S5:ineq:Jestimate}
    |\J|\leq Cs^{-3}\ld^{-3}\left(\iint e^{-2s\eta}|\check{f}|^2dxdt\right.\\\left.+\iint_{(-T, T)\times \omega_0}(s^5\ld^6 \xi^5|w|^2dxdt+s^3\ld^4\xi^3|w_x|^2dxdt+s\ld^2 \xi|w_{xx}|^2)dxdt\right).
\end{multline}
Since $w(0, x)=e^{-2s\eta(0, x)}\sigma(x)y_0'(x)$ and $|y_0'(x)|\geq r_0>0$, we get
\begin{align*}
    \int_0^L \xi^2(0, x)|w(0, x)|^2dx\geq \frac{r_0^2}{2}\int_0^L e^{-2s\eta(0, x)}\xi^2(0, x)|\sigma(x)|^2dx.
\end{align*}
Thus, we can use the last inequality to get a bound by below from identity \eqref{S5:eq:BKidentity} and then use \eqref{S5:ineq:Jestimate} along with Young's inequality to get a bound by above, resulting in
\begin{multline*}
    \int_0^L e^{-2s\eta(0, x)}\xi^2(0, x)|\sigma(x)|^2dx\lesssim s^{-5/2}\ld^{-3}\left(\iint e^{-2s\eta}|\check{f}|^2dxdt\right.\\\left.+ \iint_{(-T, T)\times \omega_0}(s^5\ld^6 \xi^5|w|^2dxdt+s^3\ld^4\xi^3|w_x|^2dxdt+s\ld^2 \xi|w_{xx}|^2)dxdt\right)
\end{multline*}
for $s\geq s_0$ and $\ld\geq \ld_0$. using the fact that $\eta$ is even in time and $w=e^{-2s\eta}\psi=e^{-2s\eta}\widehat{v}$ we have
\begin{multline}\label{S5:ineq:BK1}
    \int_0^L e^{-2s\eta(0, x)}\xi^2(0, x)|\sigma(x)|^2dx\lesssim s^{-5/2}\ld^{-3}\left(\iint_Q (e^{-2s\eta(t, x)}+e^{-2s\eta(t, L-x)})|f|^2dxdt+\M_{\omega_{0}}(v)\right)
\end{multline}
where, $\M_{\omega_0}(v)$ gathers the local terms in $v$ as follows
\begin{align}\label{S5:eq:Momega}
    \M_{\omega_{0}}(v):=\iint_{(0, T)\times \omega_{0}}(s^5 \bd{\xi}_5e^{-2s\eta}|v|^2+s^3 \bd{\xi}_3e^{-2s\eta}|v_x|^2+s\bd{\xi}_1e^{-2s\eta}|v_{xx}|^2)dxdt,
\end{align}
with $\bd{\xi}_k$ defined as $\bd{\xi}_k(t, x):=\xi^k(t, x)+\xi^k(t, L-x)$, $(t, x)\in Q$, for $k=1, 3, 5$. Observe that we just used the change of variables $x\mapsto L-x$ and that $\omega_0$ is symmetric with respect to $L/2$.

We now look at the terms involving $f$. By using the bound \eqref{S5:ineq:Kbound}, we get
\begin{align*}
    \iint_Q e^{-2s\eta(t, x)}|f|^2dxdt&=\iint_Q e^{-2s\eta(t, x)}|\sigma(x)z_{xt}-y_{xt}u-z_tu_x|^2dxdt\\
    &\lesssim \int_0^L e^{-2s\eta(0, x)}|\sigma(x)|^2dx+\int_0^T\int_0^L e^{-2s\eta}(|u|^2+|u_x|^2)dxdt.
\end{align*}
Similarly, we have
\begin{align*}
    \iint_Q e^{-2s\eta(t, L-x)}|f|^2dxdt&=\iint_Q e^{-2s\eta(t, L-x)}|\sigma(x)z_{xt}-y_{xt}u-z_tu_x|^2dxdt\\
    &\lesssim \int_0^L e^{-2s\eta(0, L-x)}|\sigma(x)|^2dx+\int_0^T\int_0^L e^{-2s\eta(t, L-x)}(|u|^2+|u_x|^2)dxdt\\
    &=\int_0^L e^{-2s\eta(0, x)}|\sigma(x)|^2dx+\int_{-T}^0\int_0^L e^{-2s\eta(t, x)}(|\widehat{u}|^2+|\widehat{u}_x|^2)dxdt,
\end{align*}
where we used that $t\in [0, T)\mapsto e^{-s\eta(t,L-x)}$ is decreasing for any $x\in [0, L]$ and the change of variables $x\mapsto L-x$. Gathering the last two inequalities we obtain
\begin{align}\label{S5:ineq:BK2}
    \iint_Q (e^{-2s\eta(t, x)}+e^{-2s\eta(t, L-x)})|f|^2dxdt\lesssim \int_0^L e^{-2s\eta(0, x)}|\sigma(x)|^2dx+\int_{-T}^T\int_0^L e^{-2s\eta}(|\widehat{u}|^2+|\widehat{u}_x|^2)dxdt.
\end{align}
It remains to estimate the second term on the right-hand side of the above inequality.

\noindent{\emph{Step 3: Second use of the Carleman estimate.}} We apply the Carleman estimate \eqref{S5:prop:carlemanomega} to the equation satisfied by $\widehat{u}$ to obtain
\begin{multline*}
    \int_{-T}^T\int_0^L e^{-2s\eta(t, x)}(|\widehat{u}|^2+|\widehat{u}_x|^2)dxdt\lesssim s^{-3}\ld^{-4}\left(\int_{-T}^T\int_0^L e^{-2s\eta}|\sigma z_{xt}|^2dxdt\right.\\\left.
    +\iint_{(-T, T)\times \omega_{0}} e^{-2s\eta}(s^5\xi^5|\widehat{u}|^2dxdt+s^3\xi^3|\widehat{u}_x|^2dxdt+s\xi|\widehat{u}_{xx}|^2)dxdt\right).
\end{multline*}
As before, we use \eqref{S5:ineq:Kbound} to bound the first term. Then, we use the definition of the symmetric extension to we rewrite the local terms at the right-hand side of the last inequality as integrals over $(0, T)\times \omega_{0}$, Thus, with $\M_{\omega_{0}}(u)$ as defined in \eqref{S5:eq:Momega}, we get 
\begin{align}\label{S5:ineq:BK3}
    \int_{-T}^T\int_0^L e^{-2s\eta(t, x)}(|\widehat{u}|^2+|\widehat{u}_x|^2)dxdt\lesssim s^{-3}\ld^{-4}\left(\int_0^L e^{-2s\eta(0, x)}|\sigma(x)|^2dx+\M_{\omega_{0}}(u)\right).
\end{align}
Set $c(s, \ld):=1-(s^{-11/2}\ld^{-7}+s^{-5/2}\ld^{-3})$. Putting together inequalities \eqref{S5:ineq:BK1}-\eqref{S5:ineq:BK2}-\eqref{S5:ineq:BK3} and then using that $x\in [0, L]\mapsto e^{-2s\eta(0, x)}\xi^2(0, x)$ is positively bounded above and below, we obtain that
\begin{align}\label{S5:ineq:BK4}
    c(s, \ld)\int_0^L|\sigma(x)|^2dx\lesssim (s^{-5/2}\ld^{-3}\M_{\omega_{0}}(v)+s^{-3}\ld^{-4}\M_{\omega_{0}}(u)),
\end{align}
for any $s\geq s_0$ and $\ld\geq \ld_0$. As $\omega_0\Subset\omega$, we note that $(t, x)\in (0, T)\times \overline{\omega_0} \mapsto s^k\bd{\xi}_k(t, x)e^{-2s\eta(t, x)}$ is bounded by above for $k=1, 3, 5$, from which follows
\begin{align*}
    \M_{\omega_{0}}(u)+\M_{\omega_{0}}(v)\lesssim \norm{y-z}_{H^1(0, T; H^2(\omega))}^2.
\end{align*}
Since $\sigma=\mu-\nu$, the proof ends by choosing $s$ and $\ld$ large enough in \eqref{S5:ineq:BK4} so that $c(s, \ld)>0$.
\end{proof}
\begin{remark}
    We point out that the regularity assumption $(y_0, \vec{h})\in \mc{Z}_{6, T}$ is not sharp. From the proof, we need that $z_{xt}$, $y_{xt}$ belong to $L^\infty(Q)$, which would follow by Sobolev embedding provided they both belong to $L^\infty([0, T], H_{\pw}^s(0, L))$, with $s>1/2$. This could be achieved if we ask $(y_0, \vec{h})\in \mc{Z}_{4+s, T}$ (with an appropriate definition for non-integers). Nevertheless, a rigorous proof will employ Tartar's nonlinear interpolation (see \cite[Section 4]{BSZ03}) and a characterization for the interpolation of spaces involving the transmission conditions as given in \cref{app:prop:interpolation-dom}. Moreover, the spaces $\mc{Z}_{s, T}$ are most likely not sharp in regards to the regularity of the boundary data, see \cref{rk:bdry-reg} above. We do not deepen in this direction as it is outside of the scope of this work.
\end{remark}


\section{Some further remarks}\label{S6}

\subsection{Boundary observability} Under the hypothesis $p_k>p_{k-1}$ with $k\in\ran{0}{N-1}$, a straightforward modification to the proof of \cref{S3:lem:weightObs} lead us to the construction of $\beta$ with \emph{observation at $x=0$}. Given $\ld>0$, we define
\begin{align}\label{S6:eq:weightobs0}
    \eta(t, x)=\dfrac{e^{\kappa\ld\norm{\beta}_{\infty}}-e^{\ld\beta(x)}}{t(T-t)}\ \text{   and   }\   \xi(t, x)=\dfrac{e^{\ld\beta(x)}}{t(T-t)},
\end{align}
for all $(t, x)\in Q$ and some $\kappa\in (1, 2)$. By following the same steps as before, we can obtain a Carleman estimate with boundary observation for the solutions of the system
\begin{align}\label{S6:eq:adj}
\left\{
\begin{array}{rll}
\vp_t+p(x)\vp_{xxx}+\vp_x = 0, & ~(t, x)\in (0, T)\times (0, L),\\
\vp(t, 0)=\vp(t, L)=\vp_{x}(t, 0)=0, & ~t\in (0, T),\\
\vp(T, x) =  \vp_T(x), & ~x\in (0, L),
\end{array}
\right.
\end{align}
coupled by the corresponding transmission conditions \eqref{eq:tc}. The Carleman estimate is the following.

\begin{proposition}\label{S6:prop:carlemanobs0}
Let $\eta$ and $\xi$ be the weight functions defined by \eqref{S6:eq:weightobs0}. Suppose that $p_k>p_{k-1}$ for all $k\in\ran{1}{N-1}$. Then there exist $s_0>0$, $\ld_0>0$ and a constant $C>0$ depending on $L$, $T$, $\rho_0$, $\rho_1$, $\norm{\beta}_{C^3([0, L]\setminus\Gamma)}$, $r$, $s_0$ and $\ld_0$ such that for all $\vp_T\in \D(\A^*)$ we have
\begin{align*}
C\iint_Q e^{-2s\eta}\big(s^5\ld^6\xi^5|\vp|^2+s^3\ld^4\xi^3|\vp_x|^2+s\ld^2\xi|\vp_{xx}|^2\big)dxdt\leq s\ld\int_{0}^T e^{-2s\eta(t, 0)}\xi(t, 0)|\vp_{xx}(t, 0)|^2dt,
\end{align*}
for any $s\geq s_0$ and $\ld\geq \ld_0,$ where $\vp$ is the solution of \eqref{S6:eq:adj} associated to $\vp_T$.
\end{proposition}

As classically done by the HUM method, under the hypothesis that $p_k>p_{k-1}$ for $k\in\ran{1}{N-1}$, the previous Carleman estimate can be combined with a dissipation estimate to obtain, for instance, the boundary null controllability of the linear KdV equation
\begin{align*}
\left\{
\begin{array}{rll}
y_t+p(x)y_{xxx}+y_x= 0, & ~(t, x)\in (0, T)\times (0, L),\\
y(t, 0)=h(t),\ y(t, L)=y_{x}(t, L)=0, & ~t\in (0, T),\\
y(0, x) =  y_0(x), & ~x\in (0, L),
\end{array}
\right.
\end{align*}
coupled by the transmission conditions \eqref{eq:tc}, with control $h\in L^2(0, T)$ and $y_0\in L^2(0, L)$. 

In regards to the boundary control to the trajectories for the constant main coefficient case, Glass and Guerrero \cite{GG08} provided a positive result with one control acting on the left point of the interval. Trying to adapt their ideas to the discontinuous case is an interesting problem, as the proof is more involved, mainly due to the several regularity technicalities related to the boundary value problem and also that it heavily uses interpolation arguments. The latter is a problem in itself in the case of discontinuous coefficients, for which results like \cref{app:prop:interpolation-dom} below may be of interest. We point out that similar issues has been faced in Parada \cite{Par24} when studying the KdV equation on a star-shaped network.

Additionally, for the problem of exact controllability when only one control is acting on the boundary, one may expect to see the critical length phenomena in the discontinuous setting as well. Thus, the problem of exact controllability for every time, for a set of lengths and length of the set of discontinuities is a wide open problem. See \cite[Proposition 4, Remark 2]{Crepeau16} where the control is acting on the Neumann boundary condition.

\subsection{Monotonicity hypothesis on the Carleman estimate} The monotonicity hypothesis on $p$, enforced trough Assumption \ref{assumM}, is crucial to obtain the Carleman estimate with main piecewise constant coefficient. Indeed, we can then construct a weight function which is piecewise monotone and satisfies the same transmission conditions as given by the main PDE under consideration, which further allows us to obtain a weighted norm for the trace terms at the interface.

Regarding applications to inverse problems, it is worth noticing that a similar Carleman estimate with boundary observation as the one used in \cite{BCCM14} can be obtained under the hypothesis $p_{k}>p_{k-1}$, $k\in\ran{1}{N-1}$ (see also \cref{S6:prop:carlemanobs0}). However, this monotonicity condition is not compatible with Hypothesis \ref{assumM}, the latter being necessary to employ the reflection trick and therefore to avoid observations in some time $T_0\in (0, T)$, as commonly found in the parabolic case. Whether one can get rid of this monotonicity hypothesis on the coefficient $p$ is an open problem. 

The main difficulty is to construct a weight function that allows us to estimate the interface terms coming from $I_{12}$ in the Carleman estimate. These terms are not necessarily zero if the weight function does not satisfy the transmission conditions. Observe that a similar difficulty is faced when establishing a Carleman estimate for the KdV equation under Colin-Ghidaglia boundary conditions, see Guilleron \cite{Gui14} and Carreño and Guerrero \cite{CG18}.


\appendix
\setcounter{theorem}{0}
\renewcommand{\thetheorem}{\Alph{section}\arabic{theorem}}

\section{}

\subsection{Inequalities toolbox} Let $I\subset \R$ be a non-empty interval and $T>0$. Let us introduce for $s\geq 0$ the Banach space
\begin{align*}
    \bX_T^s(I)=C([0, T], H^s(I))\cap L^2(0, T; H^{s+1}(I)),
\end{align*}
equipped with the natural norm. We have the following estimate, used in \cref{S2} to obtain the well-posedness of the nonlinear system \eqref{eq:NLKdV}; see \cref{prop:NLwp-Xs}.

\begin{lemma}\cite[Lemma 3.3]{BSZ03}\label{app:lem:L1Hs-estimate}
Let $s\geq 0$ be given. There exists $C>0$ such that for any $T>0$ and $u$, $v\in \bX_T^s(I)$,
\begin{align*}
    \int_0^T \norm{u(t,\cdot) v_x(t, \cdot)}_{H^s(I)}dt\leq C(T^{1/2}+T^{1/3})\norm{u}_{\bX_T^s(I)}\norm{v}_{\bX_T^s(I)}.
\end{align*}
\end{lemma}

The following result allows us to consider $yy_x$ as a source term in \eqref{eq:kdv-lin-f}-\eqref{eq:tc}.

\begin{proposition}\cite[Proposition 4.1]{Ros97}\label{app:prop:yyxH1}
Let $y\in L^2(0, T; H^1(I))$. Then $yy_x \in L^1(0, T; L^2(I))$ and the map 
$$y\in L^2(0, T; H^1(I))\longmapsto yy_x \in L^1(0, T; L^2(I))$$ is continuous and there exists $C>0$ such that
\begin{align*}
    \norm{yy_x-zz_x}_{L^1(0, T; L^2(I))}\leq C\big(\norm{y}_{L^2(0, T; H^1(I))}+\norm{z}_{L^2(0, T; H^1(I))}\big)\norm{y-z}_{L^2(0, T; H^1(I))}.
\end{align*}
\end{proposition}

\subsection{Interpolation with constraints lemma} The aim of this section is to describe the interpolation of piecewise Sobolev spaces including the transmission conditions \eqref{eq:tc-dom}. To this end, we will use Löfstrom's results \cite{Lof95, Lof97} on interpolation with constraints in a sort of black-box way. In what follows we will work with the $K$-interpolation method, for which we set $(\cdot, \cdot)_{\theta, 2}=(\cdot, \cdot)_{\theta}$, where we drop the subscript $2$ for simplicity, as we will be working only with Hilbert spaces; we refer to \cite[Chapter 1, Section 1.3]{Triebel95} or \cite[Chapter 1]{Lof97} for further details on the topic. 

Let us set $T:=(T_0, T_1, T_2)$ with $T_j: H_{\pw}^3(0, L)\to \R^{N-1}$ defined as $T_j:=(\Lambda_j^1,\ldots, \Lambda_j^{N-1})$ with
\begin{align*}
    \Lambda_j^k(u)=\al_{j,k}u^{(j)}(a_k^+)-\alpha_{j,k-1}u^{(j)}(a_k^-),
\end{align*}
where $\alpha_{j, k}:=p_k^{j/2}$, $j\in\ran{0}{2}$. These functionals are well-defined by Sobolev embedding and encode the $N-1$ transmission conditions at each order, since we can write
\begin{align}\label{def:H3kerT}
    \H_{\tc}^3(0, L)=H_{\pw}^3(0, L)\cap\ker T.
\end{align}
The interpolation result is the following.

\begin{proposition}\label{app:prop:interpolation-dom}
    For any $\theta\in (0, 1)$ with $\theta\not\in\{\tfrac{1}{6}, \tfrac{1}{2}, \tfrac{5}{6}\}$, the interpolation space $(L^2(0, L), \H_{tc}^3(0, L))_{\theta}$ consists of those $u\in H_{\pw}^{3\theta}(0, L)$ such that
    \begin{align*}
            T_j(u)=0     & \text{ for every $j\in\{0, 1, 2\}$ for which } 3\theta>j+1/2.
    \end{align*}
\end{proposition}
\begin{remark}\label{rk:interpolationBC}
    The analogous interpolation result holds if homogeneous Dirichlet and/or Neumann boundary conditions are added. For the sake of simplicity, we choose to only consider the constraints given by the transmission conditions for the analysis.
\end{remark}

We first introduce some necessary notation for the proof. Recall that for $s\geq 0$, we identify $H_{\pw}^s(0, L)=\prod_{k=0}^{N-1}H^s(I_k)$. For each $k\in\ran{0}{N-1}$, let $E_k: H^s(I_k)\to H^s(\R)$ denote the corresponding Stein extension operator (see for instance \cite[Theorem 5.24]{AF03}) and set $\mathbf{E}:=(E_0,\ldots, E_{N-1})$ which maps $H_{\pw}^s(0, L)$ into $\prod_{k=0}^{N-1}H^s(\R)\simeq H^s(\R; \R^N)$. Let us also consider the restriction map $R_k: H^s(\R)\mapsto H^s(I_k)$, and define $\mathbf{R}=(R_0,\ldots, R_{N-1})$. In particular, $\mathbf{R}\mathbf{E}=\mathrm{Id}$ on $L^2(0, L)$ and $H_{\pw}^3(0, L)$.

Let us introduce the functionals $\widetilde{\Lambda}_j^k: H^3(\R; \R^N)\to \R$ defined as
\begin{align*}
    \widetilde{\Lambda}_j^k(v)&:=\Lambda_{j}^k\mathbf{R}(v)\\
    &=\al_{j,k}\Gamma_j^{(a_k)}(\pi_{k}v)-\alpha_{j,k-1}\Gamma_j^{(a_k)}(\pi_{k-1}v)=\al_{j,k}v_{k}^{(j)}(a_k)-\alpha_{j,k-1}v_{k-1}^{(j)}(a_k),
\end{align*}
where $\pi_k$ denotes the projection into the $k$-th coordinate and $\Gamma_j^{(a_k)}(w):=w^{(j)}(a_k)$. These functionals are indeed well-defined and continuous by one dimensional Sobolev embedding.

\begin{proof}[Proof of \cref{app:prop:interpolation-dom}]
    Let us introduce $\widetilde{T}:=T\circ \mathbf{R}$. From the identity $\mathbf{R}\mathbf{E}=\mathrm{Id}_{H_{\pw}^3(0, L)}$, if $u\in\ker T$ then $\widetilde{T}(\mathbf{E}u)=T(\mathbf{R}\mathbf{E}u)=T(u)=\mathbf{0}_{\R^{N-1}}$, which yields $\mathbf{E}(\ker T)\subset \ker \widetilde{T}$. Similarly, if $v\in\ker\widetilde{T}$, then $T(\mathbf{R}v)=\widetilde{T}(v)=\mathbf{0}_{\R^{N-1}}$, hence $\mathbf{R}(\ker\widetilde{T})\subset\ker T$. By kernel compatibility, the maps
    \begin{align*}
    \begin{array}{rlcl}
        \mathbf{E}: & H_{\pw}^3(0, L)\cap\ker T&\longrightarrow& H^3(\R, \R^N)\cap\ker\widetilde{T},\\ 
        \mathbf{R}: & H^3(\R, \R^N)\cap\ker\widetilde{T}&\longrightarrow & H_{\pw}^3(0, L)\cap\ker T,
    \end{array}
    \end{align*}
    are both linear and bounded between the endpoint constrained spaces. Since $\mathbf{E}$ and $\mathbf{R}$ are also linear and bounded in $L^2(0, L)$ and $L^2(\R; \R^N)$, respectively, by retraction-coretraction's theorem \cite[Section 1.2.4]{Triebel95} applied to the constrained interpolation couples, we can write
    \begin{align}\label{eq:retraction-identity}
        (L^2(0, L), H_{\pw}^3(0, L)\cap \ker T)_{\theta}=\mathbf{R}\big((L^2(\R; \R^N), H^3(\R; \R^N)\cap \ker\widetilde{T})_{\theta}\big).
    \end{align}

    To describe the retracted space above, we will closely follow the construction in \cite[Section 2.8]{Lof97}. For $j\in\ran{0}{2}$ and $k\in\ran{1}{N-1}$, let us consider $\vartheta_{j, k}(t):=G_t*\overline{\vp_{j, k}(t, \cdot)}/\norm{\vp_{j, k}(t, \cdot)}_{L^2(\R)}^2$ with $\vp_{j, k}(t, s):=(D^{j}G_t)(a_k-s)$ and $\widehat{G_t}(\xi)=(1+t^2|\xi|^6)^{-1}$. It satisfies $\Gamma_j^{(a_k)}(\vartheta_{j, k}(t))=1$ and
    \begin{align*}
        \max\big(\norm{\vartheta_{j,k}(t)}_{L^2(\R)}, t\norm{\vartheta_{j,k}(t)}_{H^3(\R)}\big)\lesssim t^{\theta_j},
    \end{align*}
    with breakpoint $\theta_j:=\frac{j+1/2}{3}$, see \cite[Lemma 2.4]{Lof97}. For each fixed $k$, by \cite[Lemma 2.5]{Lof97} we can then construct a supporting sequence $(\psi_{j, k})_j$ to the evaluation functionals $\big(\Gamma_j^{(a_k)}\big)_j$ with breakpoints $(\theta_j)_j$. Let $e_{k}$ be the $k$th element of the canonical basis in $\R^N$ and let us take $\chi_k$ to be a smooth cutoff function equals to $1$ near $a_k$ with disjoint support for different $k$ with $k\in\ran{1}{N-1}$. In a similar spirit to \cite[Lemma 6]{Lof95}, for $j\in\ran{0}{2}$ and $k\in\ran{1}{N-1}$, let us define
    \begin{align*}
        w_{j, k}(t, \cdot):=\frac{1}{\al_{j, k}}\chi_k(\cdot)\psi_{j, k}(t, \cdot)e_{k}.
    \end{align*}
    Since $(\partial_x^\ell\psi_{j, k})(\cdot, a_k)=\delta_{j\ell}$ for $\ell\in\ran{0}{2}$, by support localization around $a_k$, this construction yields
    \begin{align*}
            \widetilde{\Lambda}_{j'}^{k'}(w_{j, k}(t))=\delta_{k, k'}\delta_{j, j'},
    \end{align*}
    and that the breaking point of $\widetilde{\Lambda}_{j}^k$ is $\theta_j$ for every $k\in\ran{1}{N-1}$, with $j\in\ran{0}{2}$. Moreover, it holds
    \begin{align*}
        \max\big(\norm{w_{j,k}(t)}_{L^2(\R; \R^N)}, t\norm{w_{j,k}(t)}_{H^3(\R; \R^N)}\big)\lesssim t^{\theta_j}.
    \end{align*}
    This yields a supporting sequence $(w_{j, k})$ for the family $(\widetilde{\Lambda}_j^k)$ in the sense of \cite[Definition 2.3]{Lof97}. In particular, \cite[Lemma 2.2]{Lof97} implies that $(\widetilde{\Lambda}_j^k)$ is a strongly independent family of functionals. Now, let us put $\widetilde{T}_\theta:=(\widetilde{T}_{0, \theta}, \widetilde{T}_{1, \theta}, \widetilde{T}_{2, \theta})$ where $\widetilde{T}_{j, \theta}=\widetilde{T}_j$ if $\theta>\theta_j$ and $\widetilde{T}_{j, \theta}=\mathbf{0}_{\R^{N-1}}$ otherwise, and define $T_\theta:=(T_{0, \theta}, T_{1, \theta}, T_{2, \theta})$ similarly using $T_j$ instead of $\widetilde{T}_j$. Since interpolation works component-wise (by retraction-coretraction), by classical interpolation of Sobolev spaces we get
    \begin{align}\label{eq:interpolation-product}
        (L^2(0, L), H_{\pw}^3(0, L))_\theta=\mathbf{R}((L^2(\R; \R^N), H^3(\R; \R^N))_{\theta})=\prod_{k=0}^{N-1}(L^2(I_k), H^3(I_k))_{\theta}=H_{\pw}^{3\theta}(0, L).
    \end{align}  
    We apply \cite[Theorem 2.4, Corollary 2.3]{Lof97} to $(\widetilde{\Lambda}_j^k)_{j, k}$ on the couple $(L^2(\R; \R^N), H^3(\R; \R^N))$, yielding
    \begin{align}\label{eq:interpolation-bigspace}
        \big(L^2(\R; \R^N), H^3(\R; \R^N)\cap \ker\widetilde{T}\big)_{\theta}=(L^2(\R; \R^N), H^3(\R; \R^N))_{\theta}\cap\ker\widetilde{T}_\theta.
    \end{align}

    To conclude, since $\mathbf{R}\mathbf{E}=\mathrm{Id}$ also holds in $H_{\pw}^{3\theta}(0, L)$ by linear interpolation, as before, away from the breakpoints $\theta_j$, we verify the kernel compatibility
    \begin{align}\label{kernelcomptheta}
        \mathbf{E}(\ker T_\theta)\subset \ker \widetilde{T}_\theta\ \text{ and }\ \mathbf{R}(\ker\widetilde{T}_\theta)\subset\ker T_\theta,
    \end{align}
    where we used that $\widetilde{T}_\theta=T_\theta\circ 
    \mathbf{R}$. On one hand, using \eqref{eq:retraction-identity} followed by \eqref{eq:interpolation-bigspace} and \eqref{kernelcomptheta}, we have
    \begin{align*}
        (L^2(0, L), H_{\pw}^3(0, L)\cap \ker T)_{\theta}&=\mathbf{R}\big((L^2(\R; \R^N), H^3(\R; \R^N)\cap \ker\widetilde{T})_{\theta}\big)\\
        &\subset\mathbf{R}\big((L^2(\R; \R^N), H^3(\R; \R^N))_{\theta}\big)\cap\mathbf{R}\big(\ker\widetilde{T}_\theta\big)\\
        &=(L^2(0, L), H_{\pw}^3(0, L))_\theta\cap\ker T_\theta.
    \end{align*}
    On the other hand, if $u\in (L^2(0, L), H_{\pw}^3(0, L))_\theta\cap\ker T_\theta$, by \eqref{kernelcomptheta} we have $v:=\mathbf{E}u$ belongs to the constrained space $H^{3\theta}(\R; \R^N)\cap\ker \widetilde{T}_\theta$. Using \eqref{eq:interpolation-bigspace}, we have $v\in (L^2(\R; \R^N), H^3(\R;\R^N)\cap\ker\widetilde{T})_\theta$, and therefore, by \eqref{eq:retraction-identity},
    \begin{align*}
        u=\mathbf{R}v\in \mathbf{R}((L^2(\R; \R^N), H^3(\R; \R^N)\cap\ker\widetilde{T})_{\theta})=(L^2(0, L), \H_{\tc}^3(0, L))_\theta.
    \end{align*}
    Recalling definition \eqref{def:H3kerT} and characterization \eqref{eq:interpolation-product}, this ends the proof.
\end{proof}

We now derive auxiliary interpolation identities when the right endpoint space is $\D(\A^*)$ as defined in \cref{S2:wp-linear}. Let us define the endpoint couple $E_0:=L^2(0, L)$ and $E_3=\D(\A^*)$, the latter endowed with the graph norm (equivalent to the $\norm{\cdot}_{H_{\pw}^3(0, L)}$-norm). For $\sigma\in (0, 3)$ away from $\big\{\tfrac{1}{2}, \tfrac{3}{2}, \tfrac{5}{2}\big\}$, let us set $E_{\sigma}:=(E_0, E_3)_{\sigma/3}$, which is well-defined by \cref{app:prop:interpolation-dom}-\cref{rk:interpolationBC} (from now on, only the former will be referenced). Using duality for real interpolation of Hilbert spaces, we have $E_{-\sigma}=E_{\sigma}'$, with duality taken with respect to $L^2$-pivot. If we set $E_{-3}:=E_3'$, we have a Hilbert triple
\begin{align*}
    E_3\hookrightarrow E_0\hookrightarrow E_{-3}.
\end{align*}
From now on, all equalities below are understood up to equivalence of norms. Since complex and real interpolation coincide (up to equivalent norms) in the Hilbert setting, by \cite[Proposition 2.1]{LM72} we have
\begin{align}\label{eq:inter-space}
    (E_{-3}, E_{3})_{1/2}=E_0.
\end{align}
By \cref{app:prop:interpolation-dom}, $E_1=(E_0, E_3)_{1/3}=H_0^1(0, L)$, hence $E_{-1}=H^{-1}(0, L)$. Next we express both $E_{-1}$ and $E_2$ as interpolation spaces associated with the couple $(E_{-3}, E_{3})$,
\begin{align*}
    E_{-1}=(E_{-3}, E_3)_{1/3}\ \text{ and }\
    E_{2}=(E_{-3}, E_3)_{5/6}.
\end{align*}
Indeed, using duality, symmetry and then reiteration for the real method \cite[Section 1.11.2, Section 1.3.3, Section 1.10.2]{Triebel95} and \eqref{eq:inter-space},
\begin{align*}
    E_{-1}=E_1'=(E_0, E_3)_{1/3}'=(E_0, E_{-3})_{1/3}=(E_{-3}, E_0)_{2/3}=(E_{-3}, (E_{-3}, E_{3})_{1/2})_{2/3}=(E_{-3}, E_3)_{1/3}.
\end{align*}
Similarly,
\begin{align*}
    E_2=(E_0, E_3)_{2/3}=((E_{-3}, E_3)_{1/2}, E_3)_{2/3}=(E_{-3}, E_3)_{5/6}.
\end{align*}
Then, by reiteration,
\begin{align*}
    (E_{-1}, E_2)_{\eta}=((E_{-3}, E_3)_{1/3}, (E_{-3}, E_3)_{5/6})_{\eta}=(E_{-3}, E_3)_{\theta(\eta)},
\end{align*}
where $\theta(\eta)=(1-\eta)\frac{1}{3}+\eta\frac{5}{6}$. Therefore, if we pick $\eta=1/3$ and $\eta=2/3$ we get
\begin{align*}
    (E_{-1}, E_2)_{1/3}=E_0=L^2(0, L)\ \text{ and }\ (E_{-1}, E_2)_{2/3}=E_1=H_0^1(0, L).
\end{align*}
Set $\H_{\tc,\mathrm{DDN}}^{2}(0,L):=E_2$, which corresponds to those $u\in H_{\pw}^2(0, L)$ such that $T_0u=T_1u=\mathbf{0}_{\R^{N-1}}$ and $u(0)=u(L)=u'(0)=0$, as precised in \cref{app:prop:interpolation-dom}. We thus have the following lemma.

\begin{lemma}\label{app:lem:interpolation-source}
    It holds
    \begin{align*}
        (H^{-1}(0, L), \H_{tc,\mathrm{DDN}}^2(0, L))_{1/3}=L^2(0, L)\ \text{ and }\ (H^{-1}(0, L), \H_{tc,\mathrm{DDN}}^2(0, L))_{2/3}=H_0^1(0, L).
    \end{align*}
    \begin{align*}
        (H^{-1}(0, L), \H_{tc,\mathrm{DDN}}^2(0, L))_{1/3}&=L^2(0, L),\\ (H^{-1}(0, L), \H_{tc,\mathrm{DDN}}^2(0, L))_{2/3}&=H_0^1(0, L).
    \end{align*}
\end{lemma}

\bibliographystyle{alpha}
\bibliography{bibliography}

\end{document}